\documentclass[11pt,a4paper]{amsart}
\usepackage{amsmath, amssymb,amsthm}
\usepackage{color, hyperref} 
\hypersetup{hidelinks,
  pdftitle={On the Iwasawa theory of modular forms with reducible and non-p-distinguished residual Galois representations},
  pdfauthor={Sheng-Chi Shih and Jun Wang}}
\usepackage[all]{xy}
\usepackage{bbm}
\usepackage{enumerate,enumitem}
\usepackage{comment}
\usepackage{mathrsfs}
\usepackage{graphicx}
\usepackage{tabularx}
\usepackage{tikz-cd} 
\usepackage{MnSymbol} 

\theoremstyle{plain}
\newtheorem{thm}{Theorem}[section]
\newtheorem{lemma}[thm]{Lemma}
\newtheorem{prop}[thm]{Proposition}
\newtheorem{cor}[thm]{Corollary}
\newtheorem{conj}[thm]{Conjecture}
\newtheorem{ex}[thm]{Example}

\theoremstyle{definition}

\theoremstyle{remark}
\newtheorem{remark}[thm]{Remark}

\DeclareMathOperator{\Hom}{Hom}

\DeclareMathOperator{\Gal}{Gal}
\DeclareMathOperator{\ord}{ord}
\DeclareMathOperator{\Char}{Char}

\DeclareMathOperator{\alg}{alg}

\DeclareMathOperator{\Tr}{Tr}

\DeclareMathOperator{\an}{an}

\DeclareMathOperator{\Sel}{Sel}

\DeclareTextCompositeCommand{\"}{OT1}{i}{\"\i}

\newcommand{\Z}{\mathbb Z}
\newcommand{\Q}{\mathbb Q}

\newcommand{\C}{\mathbb C}

\newcommand{\F}{\mathbb F}

\newcommand{\fk}{\mathfrak{k}}

\newcommand{\cC}{\mathcal{C}}

\newcommand{\cE}{\mathcal{E}}

\newcommand{\cH}{\mathcal{H}}

\newcommand{\cL}{\mathcal{L}}
\newcommand{\cO}{\mathcal{O}}

\newcommand{\rG}{\mathrm{G}}
\newcommand{\rI}{\mathrm{I}}
\newcommand{\et}{\mathrm{\acute{e}t}}

\newcommand{\gh}{\mathfrak h}
\newcommand{\gp}{\mathfrak p}
\newcommand{\gm}{\mathfrak m}
\newcommand{\gX}{\mathfrak X}

\newcommand{\isom}{\cong}
\newcommand{\embed}{\hookrightarrow}

\newcommand{\ol}{\overline}

\numberwithin{equation}{section}

\begin{document}
\raggedbottom
\title
{On the Iwasawa theory of modular forms with reducible and non-p-distinguished residual Galois representations}

\author{Sheng-Chi Shih}
\address{Fakultät für Mathematik, University of Vienna, Oskar-Morgenstern-Platz 1, A-1090 Wien, Austria.}
\email{sheng-chi.shih@univie.ac.at}
\author{Jun Wang}
\address{Institute for Advanced Study in Mathematics, Harbin Institute of Technology, 150001 Harbin, P. R. China}
\email{junwangmath@hit.edu.cn}
\date{}

\begin{abstract}
Let $f$ be the ordinary $p$-stabilization of the weight-one Eisenstein
series $E_1(\chi,\mathbbm{1})$, where $\chi$ is odd and $\chi(p)=1$.
Under a Gorenstein hypothesis on the cuspidal Hecke algebra, we prove
a factorization of the even branches of the cyclotomic $p$-adic
$L$-function of $f$ by evaluating a cup product of cyclotomic units
via local reciprocity. The odd branches follow from the Eisenstein
specialization formula of Fukaya and Kato with a minor modification
for the exceptional character considered here. We then obtain results on the $\mu$- and $\lambda$-invariants.

On the algebraic side, we compute the Selmer characteristic ideals at $f$
and the Iwasawa invariants of the higher-weight specializations.
For every nontrivial twist of $f$, the ordinary characteristic ideal
is generated by the corresponding $p$-adic $L$-function. On the
untwisted branch it has one additional factor, the cyclotomic
augmentation parameter. Imposing the strict local condition at $p$
removes this factor. For higher-weight specializations coming from
newforms of level prime to $p$, Kato's divisibility then gives the
cyclotomic main conjecture after inverting $p$.
\end{abstract}
\maketitle
\setcounter{tocdepth}{1}
\tableofcontents

\section{Introduction}\label{sec:01}
At the weight-one Eisenstein point considered here, the ordinary cuspidal
family provides a $p$-adic $L$-function even though the specialization
is not a classical cusp form. We determine its cyclotomic branches and compare them with the
Selmer groups associated with the lattice furnished by modular-curve cohomology.
The two exceptional branches exhibit the same analytic cancellation, but
only the untwisted branch requires a change from the ordinary to the
strict Selmer condition.

\subsection{The cyclotomic main conjecture in higher weight}\label{sec:main_conj}
Fix a prime $p\geq5$ and a positive integer $N$ with $p\nmid N\phi(N)$.
Write $\Q_\infty$ for the cyclotomic $\Z_p$-extension of $\Q$,
$\Gamma=\Gal(\Q_\infty/\Q)$, and $\omega_p$ and $\kappa_p$ for the
Teichm\"uller and $p$-adic cyclotomic characters, respectively. For a
field $L$, put $\rG_L=\Gal(\overline L/L)$.

Let $g$ be a normalized ordinary cuspidal eigenform of weight $k\geq2$
and level $\Gamma_1(Np)$, with nebentype of conductor $N$ or $Np$.
Let $K/\Q_p$ be a finite extension containing its Hecke eigenvalues, let $\cO$ be the ring of integers
of $K$, and let $T_g$ be a stable lattice in the representation $V_g$
in the convention of Section~\ref{sec:choice_of_lattice}. Put
\[
 A_{g,j}=(V_g/T_g)\otimes\omega_p^j,\qquad
 \Lambda_{\cO}=\cO\lsem\Gamma\rsem\simeq\cO\lsem X\rsem.
\]
The Pontryagin dual of the ordinary Selmer group
$\Sel(\Q_\infty,A_{g,j})$ is finitely generated over $\Lambda_{\cO}$
\cite[Proposition~6]{greenberg89}. For a finitely generated torsion
module $M$, write $\Char_{\Lambda_{\cO}}M$ for its characteristic
ideal. The analytic counterpart is the branch $\cL(g,\omega_p^j)$
of the cyclotomic $p$-adic $L$-function
\cite{mazur-tate-teitelbaum}; the period and branch conventions are
specified in Section~\ref{sec:p_adic_L}.
\begin{conj}[Main conjecture]\label{main_conj}
For $g$ of weight $k\geq2$ and $0\leq j\leq p-2$, the $\Lambda_{\cO}$-module $\Sel(\Q_{\infty},A_{g,j})^{\vee}$ is torsion, and we have
\begin{equation}\label{eq:main_conj}
\cL(g,\omega_p^j)\cdot \Lambda_{\cO}[1/p]=\Char_{\Lambda_{\cO}[1/p]} \Sel(\Q_{\infty},A_{g,j})^{\vee}.
\end{equation}
\end{conj}

Inverting $p$ makes the equality independent of the choice of stable
lattice \cite[p.~102]{greenberg89}. When $g$ is the ordinary
$p$-stabilization of a newform of level $N$ prime to $p$, Kato's
divisibility \cite[Proposition~17.1 and Theorem~17.4]{kato04} gives
\begin{equation}\label{eq:divisible_kato}
 \Char_{\Lambda_{\cO}[1/p]}\Sel(\Q_\infty,A_{g,j})^\vee
 \mid \cL(g,\omega_p^j)\Lambda_{\cO}[1/p].
\end{equation}
For results in the $p$-distinguished setting, see
\cite{skinner-urban,greenberg-vatsal00,bellaiche-pollack}.
Our residual representation is reducible and scalar at $p$, so it is
not $p$-distinguished. At weight one, we formulate the comparison for the fixed ordinary
filtration at $f$ separately from Conjecture~\ref{main_conj}.
For the analogous construction by specialization
at irreducible weight-one points, see
\cite{greenberg-vatsal20}.

Castella, Grossi, and Skinner \cite{castella-grossi-skinner} prove
Mazur's main conjecture for rational elliptic curves at Eisenstein primes
under their local hypotheses. Their argument uses an auxiliary imaginary
quadratic field and anticyclotomic Iwasawa theory, together with
cyclotomic Beilinson--Flach classes. Our approach is primarily cyclotomic,
combining explicit reciprocity for cyclotomic units, the abelian main
conjecture, residual Selmer control, and Kato's divisibility.

\subsection{Main results}
Let $\chi$ be an odd primitive Dirichlet character of conductor $N$
with $\chi(p)=1$. Set $\theta=\chi\omega_p$ and
$\cO=\Z_p[\theta]=\Z_p[\chi]$. The two roots of the Hecke polynomial at
$p$ of
\[
 E_1(\chi,\mathbbm{1})
 =\tfrac12 L(0,\chi)+\sum_{n\geq1}\Bigl(\sum_{d\mid n}\chi(d)\Bigr)q^n
\]
are both $1$. Write $f$ for its unique $p$-stabilized eigenform.
Although it arises from an Eisenstein series, $f$ is an ordinary
$p$-adic cusp form and lies on a unique cuspidal component through
the weight-one point \cite[Theorem~A(i), Proposition~4.7]{BDP}.
Let $\gh_{\theta,\gm}$ be the ordinary cuspidal Hecke algebra on the
$\theta$-component, localized at the maximal ideal containing the
Eisenstein ideal $I_\theta$ of $\cE(\theta,\mathbbm{1})$
(see \eqref{eq:eis_ideal}). We assume throughout that
$\gh_{\theta,\gm}$ is Gorenstein. Specializing the Mazur--Kitagawa
element at $f$ then defines $\cL(f,\omega_p^j)$ with respect to an
integral basis of each sign component.

Maksoud's arithmetic criterion provides a way to check the Gorenstein
hypothesis for quadratic $\chi$. This hypothesis holds precisely when the
$p$-class group of the field cut out by $\chi$ is generated by the
classes of primes above $p$ \cite[Theorem~1.3.1]{maksoud-exceptional}.
For a fixed $p$, the Cohen--Lenstra heuristics suggest that this hypothesis
holds in most quadratic cases. Appendix~\ref{app:maksoud-gorenstein}
explains this prediction and gives a concrete example.

The new factorization is the one on the even branches. Its proof uses
local reciprocity for a cup product of cyclotomic units. The
factorization on the odd branches is essentially due to Fukaya and
Kato \cite[Theorem~8.2.2(1)]{FK12}. To apply their argument at our
weight-one point, we identify the plus Eisenstein quotient for the
exceptional character $\theta=\chi\omega_p$.

Let $X_{Np^{\infty}}^{(\chi^{-1}\omega_p^j)}$ (resp.~
$\mathfrak{X}_{p^{\infty}}^{(\omega_p^j)}$) denote the
$\chi^{-1}\omega_p^j$-component (resp.~the $\omega_p^j$-component) of
the Iwasawa module unramified everywhere (resp.~unramified outside
$p$), as defined in \eqref{eq:iwasawa_module}. Iwasawa
\cite{iwasawa73} proved that these are finitely generated torsion
$\Lambda_{\cO}$-modules when $j$ is even. Let
$u\in1+p\Z_p$ be a fixed topological generator. We write
$\lambda^{\an}$ and $\mu^{\an}$ for the invariants of the $p$-adic
$L$-function and $\lambda^{\alg}$ and $\mu^{\alg}$ for those of the
Selmer characteristic ideal; see \eqref{eq:def_an_inv} and
\eqref{eq:def_alg_inv} for the precise definitions.

Normalize the Kubota--Leopoldt series by
$L_p(1-s,\eta)=G_\eta(u^s-1)$, with $u$ as above, and put
$Y=u(1+X)^{-1}-1$. We retain the Iwasawa units in the factorizations
to distinguish them from the constant ambiguity in the periods. Write
$N/\omega_p(N)=u^{s_N}$, set $\mathscr U_{\mathbbm{1}}(X)=XG_{\mathbbm{1}}(X)$,
and define
\[
 \mathcal U_N(X)=(1+X)^{s_N},\quad
 \mathcal V_0(X)=\mathcal U_N(X)\mathscr U_{\mathbbm{1}}(X),\quad
 \mathcal V_1(X)=-(1+X)\mathscr U_{\mathbbm{1}}(Y).
\]
All three are units of $\Lambda_{\cO}$; the notation $\sim_{\cO}$ below
denotes equality up to a \emph{constant} in $\cO^\times$.
\begin{thm}[Weight-one factorization]\label{thm:intro-weight-one-factorization}
Assume that the cuspidal Hecke algebra $\gh_{\theta,\gm}$ is Gorenstein.
\begin{enumerate}
\item If $0\leq j\leq p-2$ is even, then
\begin{equation}\label{eq:p_adic_L_fcn_wt_one_branch}
\cL(f,\omega_p^j)(X)\sim_{\cO}
\begin{cases}
\mathcal V_0(X)\dfrac{G_{\chi\omega_p}(Y)}{X},&j=0,\\[6pt]
\mathcal U_N(X)G_{\omega_p^{j}}(X)G_{\chi\omega_p^{1-j}}(Y),&j>0.
\end{cases}
\end{equation}
For $j=0$, the numerator is divisible by $X$, since
\[
G_{\chi\omega_p}(u-1)=L_p(0,\chi\omega_p)
=(1-\chi(p))L(0,\chi)=0.
\]
Thus the quotient in \eqref{eq:p_adic_L_fcn_wt_one_branch} belongs to
$\Lambda_{\cO}$. Moreover,
\begin{equation}\label{eq:intro-even-analytic-invariants}
\mu^{\an}(f,j)=0,
\qquad
\lambda^{\an}(f,j)=
\begin{cases}
\lambda\bigl(G_{\chi\omega_p}(u(1+X)^{-1}-1)\bigr)-1,&j=0,\\[4pt]
\lambda\bigl(G_{\omega_p^j}\bigr)
+\lambda\bigl(G_{\chi\omega_p^{1-j}}\bigr),&j>0.
\end{cases}
\end{equation}

\item \textup{(Odd branches, after Fukaya--Kato.)}
If $1\leq j\leq p-2$ is odd, then
\begin{equation}\label{eq:intro-odd-analytic-f}
\cL(f,\omega_p^j)(X)\sim_{\cO}
\begin{cases}
G_{\chi^{-1}\omega_p^j}(X)
G_{\omega_p^{1-j}}\bigl(u(1+X)^{-1}-1\bigr),& j\neq1,\\[4pt]
\mathcal V_1(X)\dfrac{G_{\chi^{-1}\omega_p}(X)}{1+X-u},& j=1.
\end{cases}
\end{equation}
The quotient for $j=1$ also belongs to $\Lambda_{\cO}$: the pole of the
trivial-character Kubota--Leopoldt function is cancelled by the trivial
zero of $G_{\chi^{-1}\omega_p}$. Furthermore,
\[
\mu^{\an}(f,j)=0,
\qquad
\lambda^{\an}(f,j)=
\begin{cases}
\lambda(G_{\chi^{-1}\omega_p^j})
+\lambda(G_{\omega_p^{1-j}}),&j\neq1,\\[4pt]
\lambda(G_{\chi^{-1}\omega_p})-1,&j=1.
\end{cases}
\]
\end{enumerate}

\end{thm}

The two exceptional quotients have the same explanation. At $X=0$ the
zero of $G_{\chi\omega_p}(Y)$ cancels the pole of $G_{\mathbbm{1}}(X)$;
at $X=u-1$ the zero of $G_{\chi^{-1}\omega_p}(X)$ cancels the pole
of $G_{\mathbbm{1}}(Y)$. Each cancellation lowers the corresponding
$\lambda$-invariant by one. The Selmer groups behave differently on
these two branches.
\begin{thm}[Weight-one Selmer groups]\label{thm:intro-weight-one-selmer}
Assume that the cuspidal Hecke algebra $\gh_{\theta,\gm}$ is Gorenstein.
\begin{enumerate}
\item If
$0\leq j\leq p-2$ is even, then
$\Sel(\Q_{\infty},A_{f,j})^{\vee}$ is a torsion
$\Lambda_{\cO}$-module and
\begin{equation}\label{eq:sel_gp_weight_one_nontwist}
\Char_{\Lambda_{\cO}}\Sel(\Q_{\infty},A_{f,j})^{\vee} =
\begin{cases}
\Char_{\Lambda_{\cO}} X_{Np^{\infty}}^{(\chi^{-1})},&j=0,\\
\Char_{\Lambda_{\cO}}\mathfrak{X}_{p^{\infty}}^{(\omega_p^j)}
\cdot \Char_{\Lambda_{\cO}} X_{Np^{\infty}}^{(\chi^{-1}\omega_p^j)},&j>0.
\end{cases}
\end{equation}

\item If $1\leq j\leq p-2$ is odd, then
$\Sel(\Q_{\infty},A_{f,j})^{\vee}$ is a torsion
$\Lambda_{\cO}$-module and
\begin{equation}\label{eq:intro-odd-algebraic-f}
\Char_{\Lambda_{\cO}}\Sel(\Q_{\infty},A_{f,j})^{\vee}
=
\begin{cases}
\bigl(G_{\chi^{-1}\omega_p^j}(X)
G_{\omega_p^{1-j}}(u(1+X)^{-1}-1)\bigr),& j\neq1,\\[4pt]
\left(\dfrac{G_{\chi^{-1}\omega_p}(X)}{1+X-u}\right),& j=1.
\end{cases}
\end{equation}
\end{enumerate}
The algebraic $\mu$-invariants vanish in all cases. For every
$1\leq j\leq p-2$, the analytic and algebraic $\lambda$-invariants
agree. The resulting integral main-conjecture-type equality is recorded
in Corollary~\ref{cor:intro-main-conjecture}(2).
\end{thm}

\begin{remark}\label{rem:weight-one-algebraic-asymmetry}
For $j=0$,
the ordinary Selmer group of $f$ has characteristic ideal
$\Char_{\Lambda_{\cO}}X_{Np^\infty}^{(\chi^{-1})}$ and retains the
extra linear factor. Thus the ordinary Selmer group on the branch
$(f,0)$ does not satisfy the analogous equality in
Corollary~\ref{cor:intro-main-conjecture}(2). In higher weight, the
global invariant term in residual Selmer control lowers the algebraic
$\lambda$-invariant by one, as in \eqref{eq:inv_higher_wt}. On
the exceptional odd branch $j=1$, the same linear correction already
occurs in the Selmer group of $f$ itself.
For $(f,0)$, one can instead impose the strict local condition at $p$;
Corollary~\ref{cor:strict-main-f0} shows that the resulting strict Selmer
group satisfies the expected analytic equality.
\end{remark}

Recall that the Iwasawa main conjecture for abelian extensions of $\Q$, proved by Mazur and Wiles \cite{mazur-wiles}, asserts that
\begin{equation}\label{eq:classical_main_conj}
\Char_{\Lambda_{\cO}}(X_{Np^{\infty}}^{(\chi^{-1}\omega_p^j)})=(G_{\chi\omega_p^{1-j}}(u(1+X)^{-1}-1))
\end{equation}
for even $j$. For positive even $j$, it also gives
\begin{equation}\label{eq:classical-main-conj-p-ramified}
\Char_{\Lambda_{\cO}}(\gX_{p^{\infty}}^{(\omega_p^j)})
=(G_{\omega_p^j}(X)).
\end{equation}
For $j=0$, the trivial-character function is instead the
pseudo-measure described in \eqref{eq:trivial-zeta-pole-at-zero}; this
is the source of the exceptional correction by $X$.
Together with the corresponding $p$-ramified and unramified
formulations on the odd branches, this identifies the analytic and
algebraic invariants of every higher-weight specialization.

\begin{thm}[Higher-weight specializations]\label{thm:intro-higher-weight}
Let $g$ be a classical ordinary specialization of weight $k\geq2$ on a
cuspidal Hida component corresponding to a minimal prime of
$\gh_{\theta,\gm}$. As in Section~\ref{sec:p_adic_L_fcn_higher_wt},
assume that $g$ is either the ordinary $p$-stabilization of a newform
of level $N$ or a newform of level $Np$ with character of conductor
$Np$. Let $\cO_g$ be the ring of integers of its Hecke field. For
$0\leq j\leq p-2$, choose the lattice $T_{g,j}$ defined in
Section~\ref{sec:choice_of_lattice} and assume cuspidal gorenstein condition. Then
\[
\mu^{\an}(g,j)=\mu^{\alg}(A_{g,j})=0.
\]
\begin{enumerate}
\item If $j$ is even, then
\begin{equation}\label{eq:an_inv_higher_wt}
\lambda^{\an}(g,j)=
\begin{cases}
\lambda\bigl(G_{\chi\omega_p}(u(1+X)^{-1}-1)\bigr)-1,&j=0,\\[4pt]
\lambda\bigl(G_{\omega_p^j}(X)\bigr)
+\lambda\bigl(G_{\chi\omega_p^{1-j}}(u(1+X)^{-1}-1)\bigr),&j>0,
\end{cases}
\end{equation}
and
\begin{equation}\label{eq:inv_higher_wt}
\lambda^{\alg}(A_{g,j})=
\begin{cases}
\lambda(\Char_{\Lambda_{\cO}}X_{Np^{\infty}}^{(\chi^{-1})})-1,&j=0,\\
\lambda(\Char_{\Lambda_{\cO}}\mathfrak{X}_{p^{\infty}}^{(\omega_p^j)})
+\lambda(\Char_{\Lambda_{\cO}}X_{Np^{\infty}}^{(\chi^{-1}\omega_p^{j})}),&j>0.
\end{cases}
\end{equation}
In particular, the abelian main conjecture implies
$\lambda^{\an}(g,j)=\lambda^{\alg}(A_{g,j})$.

\item If $j$ is odd, then
\begin{equation}\label{eq:intro-odd-invariants-g}
\lambda^{\an}(g,j)=\lambda^{\alg}(A_{g,j})=
\begin{cases}
\lambda\bigl(G_{\chi^{-1}\omega_p^j}\bigr)
+\lambda\bigl(G_{\omega_p^{1-j}}\bigr),& j\neq1,\\[4pt]
\lambda\bigl(G_{\chi^{-1}\omega_p}\bigr)-1,& j=1.
\end{cases}
\end{equation}
\end{enumerate}
If, in addition, $g$ is the ordinary $p$-stabilization of a newform of
level $N$ prime to $p$, then for every $0\le j\le p-2$ one has
\begin{equation}\label{eq:intro-main-g}
\cL(g,\omega_p^j)\Lambda_{\cO_g}[1/p]
=\Char_{\Lambda_{\cO_g}[1/p]}
\Sel(\Q_{\infty},A_{g,j})^{\vee}.
\end{equation}
\end{thm}

\begin{cor}\label{cor:intro-main-conjecture}
Under the standing Gorenstein hypothesis, the following hold.
\begin{enumerate}
\item Let $g$ be a higher-weight specialization as in
Theorem~\ref{thm:intro-higher-weight}. If $g$ is the ordinary
$p$-stabilization of a newform of level $N$ prime to $p$, then
Conjecture~\ref{main_conj} holds for every $0\leq j\leq p-2$.
\item For $1\leq j\leq p-2$, the analogue of the
main-conjecture equality for the weight-one form $f$ holds integrally:
\[
\cL(f,\omega_p^j)\Lambda_{\cO}
=\Char_{\Lambda_{\cO}}\Sel(\Q_\infty,A_{f,j})^\vee.
\]

\item On the trivial branch of
$f$, one instead has
\begin{equation}\label{eq:intro-f0-extra-factor}
\Char_{\Lambda_{\cO}}\Sel(\Q_\infty,A_{f,0})^\vee
=X\,\cL(f,\mathbbm{1})\Lambda_{\cO}.
\end{equation}
Thus the ordinary Selmer group of $f$ fails to satisfy the analogous
main-conjecture equality at $j=0$ by precisely the factor $X$.
On the other hand, the strict Selmer group satisfies
\begin{equation}\label{eq:intro-strict-main-f0}
\Char_{\Lambda_{\cO}}\Sel^{\mathrm{st}}(\Q_\infty,A_{f,0})^\vee
=\cL(f,\mathbbm{1})\Lambda_{\cO}.
\end{equation}
Thus the strict local condition removes exactly the algebraic factor
corresponding to the even trivial zero.
\end{enumerate}
\end{cor}

Our even-branch factorization proves the unit-content conclusion of
Pollack--Wake \cite[Conjecture~4.9(2)]{pollack-wake} in the exceptional
setting considered here. Indeed, a unit coefficient at $f$ lifts to a unit
in the local cuspidal Hecke algebra, so every cuspidally normalized even
branch has unit content. The algebraic vanishing in
Theorem~\ref{thm:intro-higher-weight}, under its stated hypotheses, also
shows that the cohomological lattice of
Section~\ref{sec:choice_of_lattice} attains the minimal possible algebraic
$\mu$-invariant, namely zero. This specifies the lattice normalization,
which matters because algebraic $\mu$-invariants can vary with the stable
lattice in the Eisenstein setting.

\subsection{Proof strategy and organization}\label{sec:idea}
Section~\ref{sec:p_adic_L} treats the analytic side. For odd twists,
the congruence-module calculation of Shih--Wang and Ohta's pairing
supply the Eisenstein quotient needed in Fukaya--Kato's specialization
argument. For even twists, Sharifi's map sends the modular-symbol
measure to a cup product of cyclotomic units. We project to the local
Brauer quotient and evaluate the resulting measure by explicit
reciprocity. A nontrivial even branch has $\mu=0$ by
Ferrero--Washington; this shows that the specialized coefficient map
is primitive, so the computation determines the measure on the
original lattice. Reduction of the two-variable element then transfers
the analytic invariants to the higher-weight specializations.

Section~\ref{sec:sel_gp} fixes the cohomological lattice and the ordinary
filtration. Under the Gorenstein assumption, one can choose generators of the
two global character lines at weight one so that the ordinary line is
generated by their sum. In Section~\ref{sec:alg_inv} this description
turns the ordinary condition into a relation between two abelian
cohomology classes. We treat $j=0$ first, where replacing inertia by
the decomposition group at $p$ removes the augmentation factor, and
then treat the positive even and odd twists. For higher weights the
residual representation and its ordinary line are obtained by scalar
extension from the same residual pair. Comparing the residual Selmer
group with the uniformizer-torsion of the ordinary Selmer group gives
the algebraic invariants. When $g$ comes from level $N$, equality of
$\lambda$-invariants and Kato's divisibility \eqref{eq:divisible_kato}
give equality of characteristic ideals after inverting $p$.
\subsection*{Acknowledgments}
The first author (S.S.) was supported by the Austrian Science Fund (FWF)
under Grant No.~START-Prize Y966. He thanks A.~Betina, M.~Dimitrov,
A.~Maksoud, and C.~Wang-Erickson for stimulating discussions and helpful
comments on an earlier version of the manuscript.

The second author (J.W.) is supported by the National Natural Science
Foundation of China under Grant No.~12331004. Part of the research for
this article was carried out during visits to the Morningside Center of
Mathematics, Westlake University, and the Yau Mathematical Sciences Center.
He thanks these institutions for their hospitality. He also thanks
A.~Maksoud for helpful discussions during the Sino-French Conference on
Number Theory, held from June~28 to July~4, 2026. These discussions led
the authors to identify an error in the computation of the odd branches
in an earlier version of this paper. The authors thank the organizers
for a wonderful conference.

\section{Analytic Iwasawa invariants}\label{sec:p_adic_L}
Let the notation be as in the Introduction. We fix once and for all
embeddings $\iota_p:\bar{\Q}\embed\bar{\Q}_p$ and
$\iota_\infty:\bar{\Q}\embed\C$. We may then regard algebraic modular
forms as $p$-adic modular forms via $\iota_p$.

This section fixes the analytic normalization used in the two
factorization arguments. We first recall the interpolation properties
of the $p$-adic $L$-functions of ordinary modular forms of weight
$k\geq2$, following \cite[Section~4.5.1]{FK12}. We then explain how the
Mazur--Kitagawa two-variable element defines the $p$-adic $L$-function
of $f$ under the Gorenstein hypothesis, and we record our conventions
for Kubota--Leopoldt power series and analytic Iwasawa invariants.

\subsection{\texorpdfstring{$p$-adic $L$-functions for ordinary cusp
forms of weight $k\geq2$}{p-adic L-functions for ordinary cusp forms of
weight at least two}}\label{sec:p_adic_L_fcn_higher_wt}
Let $g=\sum_{n\geq1}a_n(g)q^n$ be a normalized $p$-ordinary
eigenform of weight $k\geq2$ in
$S_k(\Gamma_1(Np),\varphi,\bar\Q_p)$, where $\varphi$ is a primitive
character of conductor $N$ or $Np$. For the interpolation formulas
below, if the conductor is $N$, take $g$ to be the ordinary
$p$-stabilization of the primitive eigenform of level $N$. Let
$K_g:=\Q_p(a_n(g)\mid n\in\Z_{\geq1})$ and let $\cO_g$ be its ring of
integers.

Choose integral periods $\Omega_g^\pm$ for $g$. The choice of periods
on our Hida component is specified below. We view normalized
algebraic values in $\bar\Q_p$ using the fixed embeddings.
The associated measure $\cL(g)\in\cO_g\lsem\Z_p^\times\rsem$
is characterized by the following interpolation property. For every $r=1,\ldots,k-1$ and every primitive
Dirichlet character
$\psi:(\Z/p^n\Z)^\times\to\bar\Q_p$, let
$\cL(g,\psi,r)$ be the image of $\cL(g)$ under the homomorphism
\[
 \psi\kappa^{r-1}:
 \cO_g\lsem\Z_p^\times\rsem\longrightarrow\bar\Q_p
\]
determined by $[a]\mapsto a^{r-1}\psi(a)$. If $\psi$ is nontrivial,
then
\begin{equation}\label{eq:special_value_nontrivial}
 \begin{aligned}
  \cL(g,\psi,r)
  ={}&(r-1)!\,p^{n(r-1)}a_p(g)^{-n}\tau(\psi)
  (-2\pi i)^{1-r}\\
  &\qquad\cdot\frac{L(g,\psi,r)}{\Omega_g^\pm}
  \in\bar\Q_p.
 \end{aligned}
\end{equation}
Here the sign is $\pm=(-1)^{r-1}\psi(-1)$, and $\tau(\psi)$ denotes
the Gauss sum of $\psi$. If $\psi$ is trivial, then
\begin{equation}\label{eq:special_value_trivial}
 \begin{aligned}
  \cL(g,\mathbbm{1},r)
  ={}&(r-1)!\bigl(1-p^{r-1}a_p(g)^{-1}\bigr)
  \bigl(1-\varphi(p)p^{k-r-1}a_p(g)^{-1}\bigr)\\
  &\qquad\cdot(-2\pi i)^{1-r}\frac{L(g,r)}{\Omega_g^\pm}
  \in\bar\Q_p.
 \end{aligned}
\end{equation}
Here $L(g,\psi,r)$ denotes the twisted complex value in the
normalization of \cite[\S~4.5.1]{FK12}.

We distinguish the evaluation of a measure from our indexing of its
branches. For a measure $F$ and a finite-order character $\psi_p$ of
$1+p\Z_p$, define
\begin{equation}\label{eq:branch-evaluation-convention}
 F(\omega_p^j)\bigl(\psi_p(u)-1\bigr)
 =\operatorname{ev}_{\omega_p^{-j}\psi_p^{-1}}(F),
 \qquad \operatorname{ev}_\rho([a])=\rho(a).
\end{equation}
Equivalently, if $F^{\mathrm{std}}(\omega_p^i)$ denotes the branch
obtained by direct evaluation on $\omega_p^i$, then
\begin{equation}\label{eq:contragredient-versus-standard-branch}
 F(\omega_p^j)(X)=F^{\mathrm{std}}(\omega_p^{-j})
          \bigl((1+X)^{-1}-1\bigr).
\end{equation}
In particular, we use the usual evaluation map, with branch
coordinates that incorporate inversion on both $\Delta$ and $\Gamma$.

\subsection{\texorpdfstring{Mazur--Kitagawa two-variable $p$-adic
$L$-functions}{Mazur--Kitagawa two-variable p-adic L-functions}}
\label{sec:MK_p_adic_L}
Let the notation be as in the previous subsection. To keep the two
variables distinct, let $\Lambda_{\cO}^{\mathrm{wt}}$ denote the Hida
weight algebra; the cyclotomic algebra introduced in the Introduction
continues to be denoted by $\Lambda_{\cO}=\cO\lsem X\rsem$. Although
these two algebras are abstractly isomorphic, no identification between
their variables will be made. We write
$\Lambda_{\cO}^{\mathrm{wt}}\simeq\cO\lsem T\rsem$ in the weight
coordinate of \cite[Section~3.2]{shih-wang}; in this coordinate, the
exceptional weight-one prime is generated by $T+1-u$.

Set
\[
 H:=\varprojlim_r
 H^1_{\et}(X_1(Np^r)_{/\ol{\Q}},\Z_p)^{\ord},
\]
where the superscript ``\textit{ord}'' denotes the ordinary part with
respect to the dual Hecke operator $U_p^*$. Let
$\gh:=\Z_p[T_n^*\mid n\geq1]$ be the universal ordinary Hecke algebra
acting on $H$. By Poincar\'{e} duality and the comparison isomorphism
between Betti and \'{e}tale cohomology, one has an isomorphism of
$\Z_p[\rG_{\Q}]$-modules
\begin{equation}\label{eq:coho_relative_homo}
H^1_{\et}(X_1(Np^r)_{/\ol{\Q}},\Z_p)(1) \isom H_1(X_1(Np^r)(\C),\Z_p).
\end{equation}
We use the corresponding untwisted cohomology notation for modular
symbols, retaining the change of sign under complex conjugation. A
general cusp-to-cusp symbol is a relative homology class; the
Mazur--Kitagawa combination below belongs to the compact-curve
submodule.

Following Kitagawa \cite{kitagawa}, we interpolate in both the weight
and cyclotomic variables. As in \cite[\S~4.4.2]{FK12}, define the Mazur--Kitagawa
two-variable $p$-adic $L$-function by
\[
 \mathcal{L}
 :=\left(\sum_{a\in(\Z/p^r\Z)^{\times}}
 U^*(p)^{-r}\{\infty,a/p^r\}\otimes[a]\right)_{r\geq1}
 =\left(\sum_{a\in(\Z/p^r\Z)^{\times}}
 U^*(p)^{-r}[Na:1]_r\otimes[a]\right)_{r\geq1}.
\]
By \cite[Lemma~3.1]{Sha}, the system $([Na:1]_r)_{r\geq1}$
belongs to $H$, and hence $\cL\in H\lsem\Z_p^\times\rsem$. For a
Dirichlet character $\eta$ of conductor $N$ or $Np$, let $H_\eta$ be
the $\eta$-component of $H$, on which the diamond operators act through
$\eta$, and let $\cL_\eta$ be the image of $\cL$ in $H_\eta$. If
$\gm$ is a maximal ideal of the cuspidal Hecke algebra $\gh_\eta$, let
$H_{\eta,\gm}$ be the localization at $\gm$, and let
$\cL_{\eta,\gm}$ be the corresponding image. We write
\[
 \cL_{\eta,\gm}(\omega_p^j)
 \in H_{\eta,\gm}\lsem1+p\Z_p\rsem
 \simeq H_{\eta,\gm}\lsem X\rsem
\]
for its $\omega_p^j$-branch.

If $\gh_{\eta,\gm}$ is Gorenstein, Ohta's perfect duality implies that
the two eigenspaces for complex conjugation are each free of rank one
over $\gh_{\eta,\gm}$; see \cite[Section~1.6]{FK12}. Fix
$\gh_{\eta,\gm}$-linear isomorphisms
\begin{equation}\label{eq:isom_H_h}
 \iota_\eta^\pm:H_{\eta,\gm}^\pm
 \xrightarrow{\;\sim\;}\gh_{\eta,\gm}.
\end{equation}
Each character branch of $\cL_{\eta,\gm}$ lies in one of these sign
components. From now on, the same notation
$\cL_{\eta,\gm}(\omega_p^j)$ denotes its scalar image under the
appropriate map in \eqref{eq:isom_H_h}. Changing either
sign basis multiplies all branches of that sign by a unit. This is
exactly the usual ambiguity in the corresponding canonical period. With the
Tate-twist convention in \eqref{eq:coho_relative_homo}, the even
cyclotomic part lies in $H_{\eta,\gm}^{-}$ and the odd cyclotomic part
lies in $H_{\eta,\gm}^{+}$; thus all even branches use the fixed map
$\iota_\eta^{-}$, and all odd branches use $\iota_\eta^{+}$.

For every modular form $g$ of weight $k\geq2$ in a cuspidal Hida
family corresponding to a minimal prime of $\gh_{\eta,\gm}$, these
choices allow us to take $\Omega_g^\pm$ to be the canonical periods
defined in \cite[(5)]{vatsal}; see also
\cite[Section~3.3]{emerton-pollack-weston}. These periods are well
defined up to $p$-adic units. Define the specialization at the fixed tame level
by
\[
 \cL^{(N)}(g,\omega_p^j)
 :=g_*\bigl(\cL_{\eta,\gm}(\omega_p^j)\bigr)
 \in\cO_g\lsem X\rsem,
\]
where $g_*$ applies the specialization
$\gh_{\eta,\gm}\to\cO_g$ coefficientwise. Its kernel on coefficients
is denoted by $\gp_g$. This specialization
may be imprimitive at a prime dividing $N$ if the primitive tame level
of $g$ is smaller than $N$. If $g$ is new and has exact tame level $N$, then
$\cL^{(N)}(g,\omega_p^j)$ agrees with the primitive branch
$\cL(g,\omega_p^j)$ up to an element of $\cO_g^\times$; in general,
the two differ by the standard missing tame Euler factors.

We now take $\eta=\theta=\chi\omega_p$. Recall that $f$ is the unique
ordinary $p$-stabilization of $E_1(\chi,\mathbbm{1})$. We retain the
notation $\cO=\Z_p[\chi]$ and set
$\Lambda:=\Lambda_{\cO}=\cO\lsem X\rsem$. Let $\pi$ be a uniformizer
of $\cO$, and let $\gm$ be the maximal ideal of $\gh_\theta$
containing the Eisenstein ideal $I_\theta$ attached to the Eisenstein
family $\cE(\theta,\mathbbm{1})$. The ideal $I_\theta$ is generated by
\begin{equation}\label{eq:eis_ideal}
 U_q^*-1\quad(q\mid Np),
 \qquad
 T_\ell^*-1-\ell\langle\ell\rangle^{-1}\quad(\ell\nmid Np).
\end{equation}
Since the weight-one specialization of $\cE(\theta,\mathbbm{1})$ is
$E_1(\chi,\mathbbm{1})$, the prime $\gp_f$ contains $I_\theta$.

From now on, assume that $\gh_{\theta,\gm}$ is Gorenstein. Following
\cite{greenberg-vatsal20}, define the $\omega_p^j$-branch of the
$p$-adic $L$-function of $f$ by
\begin{equation}\label{eq:def_padic_Lfcn_of_f}
 \cL(f,\omega_p^j)
 :=\bigl(\cL_{\theta,\gm}(\omega_p^j)\bmod\gp_f\bigr)
 \in\Lambda.
\end{equation}
This specializes the fixed cohomological measure at the point $f$;
changing the corresponding sign basis in \eqref{eq:isom_H_h} changes
the result only by a constant in $\cO^\times$.

\subsection{Kubota--Leopoldt power series and Iwasawa invariants}\label{sec:KL-conventions}

For an even Dirichlet character $\eta$, recall that the
Kubota--Leopoldt $p$-adic $L$-function $L_p(s,\eta)$ satisfies
\begin{equation}\label{eq:interpolation_kubota_leopoldt}
L_p(1-n,\eta)=L(1-n,\eta\omega_p^{-n})(1-\eta\omega_p^{-n}(p)p^{n-1})
\end{equation}
for all $n\in\Z_{\geq1}$. If $\eta$ is of type $S$
(i.e.,~$F_\eta\cap\Q_\infty=\Q$, where $F_\eta$ is the field cut out
by $\eta$), it also admits the power-series expression
\begin{equation}\label{eq:kubota_leopoldt_power_series}
L_p(1-s,\eta)=G_{\eta}(u^s-1).
\end{equation}
For $\eta\neq\mathbbm{1}$, one has $G_\eta(X)\in\Lambda_{\cO}$.
For the trivial character, $G_{\mathbbm{1}}(X)$ is interpreted as a
pseudo-measure in the total quotient ring of $\Lambda_{\cO}$; it has a
simple pole at $X=0$ and satisfies
\begin{equation}\label{eq:trivial-zeta-pole-at-zero}
XG_{\mathbbm{1}}(X)\in\Lambda_{\cO}^{\times}.
\end{equation}
We write
\[
 \mathscr U_{\mathbbm{1}}(X):=XG_{\mathbbm{1}}(X)
 \in\Lambda_{\cO}^{\times}.
\]
See \cite[Chapter~5]{washington-GTM83} for these standard properties.
For every character $\psi$ of type $W$, one has
\begin{equation}\label{eq:KL_padic_L_typeW}
 G_{\eta\psi}(X)=G_{\eta}(\psi(u)(1+X)-1)
\end{equation}
after extending scalars to the ring generated by the values of $\psi$;
here $F_\psi\subset\Q_\infty$. Our assumptions imply that every
base character $\eta$ used below is of type $S$; the wild characters
$\psi_p$ are of type $W$ and are handled by
\eqref{eq:KL_padic_L_typeW}.

By the Weierstrass preparation theorem
\cite[Theorem~7.3]{washington-GTM83}, every nonzero element
$F(X)\in\Lambda$ admits a unique factorization
\[
 F(X)=\pi^n f(X)u(X),
\]
where $n\in\Z_{\geq0}$, $u(X)\in\Lambda^\times$, and $f(X)$ is a
distinguished polynomial. The $\lambda$- and $\mu$-invariants of
$F(X)$ are defined by
\begin{equation}\label{eq:invariant_power_series}
 \lambda(F(X))=\deg(f(X)),
 \qquad
 \mu(F(X))=n.
\end{equation}
For any modular form $h$ whose branches have been normalized as above,
set
\begin{equation}\label{eq:def_an_inv}
 \lambda^{\an}(h,j):=\lambda\bigl(\cL(h,\omega_p^j)\bigr),
 \qquad
 \mu^{\an}(h,j):=\mu\bigl(\cL(h,\omega_p^j)\bigr).
\end{equation}

\subsection{Even branches via local cup products}\label{sec:local_cup_prod}

We now prove the even factorization. The argument has three steps. We
first place the Mazur--Kitagawa element in the rank-one local
cohomology module, then compute its character specializations by an
explicit reciprocity law, and finally remove the possible common
integrality factor. Only after these steps do we return to the
$p$-adic $L$-function of $f$.

For the fixed integer $N$ from the Introduction, let $L_\infty$ and $M_\infty$
be the maximal abelian pro-$p$ extensions of
$\Q(\zeta_{Np^\infty}):=\bigcup_{n\geq1}\Q(\zeta_{Np^n})$ that are,
respectively, unramified everywhere and unramified outside $p$. Set
\begin{equation}\label{eq:iwasawa_module}
 X_{Np^\infty}:=\Gal(L_\infty/\Q(\zeta_{Np^\infty})),
 \qquad
 \mathfrak X_{Np^\infty}:=\Gal(M_\infty/\Q(\zeta_{Np^\infty})).
\end{equation}
Let $L_{\infty,p}\subset L_\infty$ be the maximal subextension in
which every prime above $p$ splits completely, and set
$X_{Np^\infty,p}:=\Gal(L_{\infty,p}/\Q(\zeta_{Np^\infty}))$.
These groups carry a natural conjugation action of
\[
\Gal(\Q(\zeta_{Np^{\infty}})/\Q)\simeq
 (\Z/Np\Z)^{\times}\times(1+p\Z_p).
\]
For a Dirichlet character $\eta$ of $(\Z/Np\Z)^\times$, let
$X_{Np^\infty}^{(\eta)}$ be the submodule on which
$(\Z/Np\Z)^\times$ acts through $\eta$. Define
$\mathfrak X_{Np^\infty}^{(\eta)}$ and
$X_{Np^\infty,p}^{(\eta)}$ similarly.

Sharifi's map \cite[Proposition~5.7]{Sha}
\[
 \varpi_{\theta}\colon H_{\theta}^-(1)
 \longrightarrow
 \varprojlim_r H^2\!\left(
   \Z[\zeta_{Np^r},\tfrac{1}{p}],\Z_p(2)
 \right)_{\theta}^+
 =:S_{\theta}
\]
sends the compatible system of Manin symbols $([u:v]^+_r)_{r\geq1}$
to the compatible system of cup products
$((1-\zeta_{Np^r}^u,1-\zeta_{Np^r}^v)_r)_{r\geq1,\theta}$.
Moreover, Fukaya--Kato \cite[Theorem~5.2.3]{FK12} proved that
$I_{\theta}H_{\theta}^-(1)$ is contained in the kernel of
$\varpi_{\theta}$. Equivalently, $\varpi_{\theta}$ factors through
$H_{\theta}^-(1)/I_{\theta}H_{\theta}^-(1)$.

Set
\[
 \gp_1:=\gp_f\cap\Lambda_{\cO}^{\mathrm{wt}}=(T+1-u).
\]
This is the height-one prime corresponding both to the exceptional-zero
factor of the Eisenstein power series and to the weight-one
specialization. Let $\xi_\theta\in\Lambda_{\cO}^{\mathrm{wt}}$ be the
Eisenstein power series, and put
\[
 \gm_\Lambda:=\gm\cap\Lambda_{\cO}^{\mathrm{wt}},
 \qquad
 \mathscr A:=(\Lambda_{\cO}^{\mathrm{wt}})_{\gm_\Lambda},
 \qquad
 S_{\theta,\gm_\Lambda}
 :=S_\theta\otimes_{\Lambda_{\cO}^{\mathrm{wt}}}\mathscr A.
\]
The standard Eisenstein-quotient identification
\begin{equation}\label{eq:eisenstein-quotient-weight-algebra}
 \gh_{\theta,\gm}/I_\theta
 \simeq\mathscr A/(\xi_\theta)
\end{equation}
is the localization of \cite[(3.10)]{shih-wang}. Since $\varpi_\theta$
kills $I_\theta H^-_\theta(1)$, it first descends to the Eisenstein
quotient. Localizing this quotient and using
\eqref{eq:eisenstein-quotient-weight-algebra} gives a well-defined map
\[
 \varpi_{\theta,\gm}\colon
 H^-_{\theta,\gm}(1)/I_\theta H^-_{\theta,\gm}(1)
 \longrightarrow S_{\theta,\gm_\Lambda}.
\]
Indeed, on the image of the map, the Hecke action factors through the
Eisenstein quotient; every element outside $\gm$ maps to a unit of
the localized quotient.

Under \eqref{eq:eisenstein-quotient-weight-algebra},
$\gp_f/I_\theta$ corresponds to $\gp_1/(\xi_\theta)$. Thus, for
$a\in\gp_f$, there is a $b\in\gp_1\mathscr A$ such that
$a-b\in I_\theta\gh_{\theta,\gm}$. Since
$\varpi_{\theta,\gm}$ is $\mathscr A$-linear,
\[
 \varpi_{\theta,\gm}(ax)=\varpi_{\theta,\gm}(bx)
 =b\varpi_{\theta,\gm}(x)
 \in\gp_1S_{\theta,\gm_\Lambda}.
\]
Consequently, $\varpi_{\theta,\gm}$ induces a well-defined homomorphism
\begin{equation}\label{eq:bar-varpi-before-local-quotient}
\bar\varpi_{\theta}:
H^-_{\theta,\gm}(1)/\gp_fH^-_{\theta,\gm}(1)
\longrightarrow
S_{\theta,\gm_\Lambda}/\gp_1S_{\theta,\gm_\Lambda}.
\end{equation}
Localizing the exact sequence of \cite[Section~3.1]{shih-wang} at
$\gm_\Lambda$ gives
\[
0\longrightarrow X_{Np^{\infty},p}^{(\chi)}(1)_{\gm_\Lambda}
\longrightarrow S_{\theta,\gm_\Lambda}
\longrightarrow S'_{\chi}(1)\longrightarrow0,
\]
where
\[
 S'_\chi=\left(\bigoplus_{v\mid p}^{0}\Z_p\right)^{(\chi)}
 \simeq\cO.
\]
Here $\bigoplus^0$ denotes the kernel of the sum map over the primes
above $p$; the localization does not change $S'_\chi(1)$ because this
module is supported at $\gm_\Lambda$. The prime $\gp_1$ annihilates
$S'_{\chi}(1)$. Moreover, \cite[Lemma~3.1]{shih-wang} implies that
\[
 X_{Np^{\infty},p}^{(\chi)}(1)_{\gm_\Lambda}/
 \gp_1X_{Np^{\infty},p}^{(\chi)}(1)_{\gm_\Lambda}
\]
is finite. Reduction modulo $\gp_1$ therefore gives a natural
surjection
\begin{equation}\label{eq:local-quotient-q-chi}
 q_{\chi}:S_{\theta,\gm_\Lambda}/\gp_1S_{\theta,\gm_\Lambda}
 \twoheadrightarrow S'_{\chi}(1)
\end{equation}
with finite kernel. Define
\begin{equation}\label{eq:bar-varpi-prime}
 \bar\varpi'_{\theta}:=q_{\chi}\circ\bar\varpi_{\theta}.
\end{equation}

\begin{lemma}\label{cup_prod_modulo_pf}
Let the notation be as above. If the cuspidal Hecke algebra
$\gh_{\theta,\gm}$ is Gorenstein, then $\bar\varpi_{\theta}$ is
injective, and the composite $\bar\varpi'_{\theta}$ is an injective
$\cO$-linear map
\[
 H^-_{\theta,\gm}(1)/\gp_fH^-_{\theta,\gm}(1)
 \longrightarrow S'_{\chi}(1).
\]
\end{lemma}

\begin{proof}
The character $\theta=\chi\omega_p$ is primitive and even of conductor
$Np$. By the Betina--Dimitrov--Pozzi identification recorded in the
proof of \cite[Proposition~3.5]{shih-wang},
\[
 (\gh_{\theta,\gm})_{\gp_f}\simeq
 (\Lambda_{\cO}^{\mathrm{wt}})_{\gp_1}.
\]
Under this isomorphism,
$I_\theta(\gh_{\theta,\gm})_{\gp_f}$ corresponds to
$\gp_1(\Lambda_{\cO}^{\mathrm{wt}})_{\gp_1}$. Here the localization on
the left is at the fixed height-one prime $\gp_f$ above $\gp_1$. Let
$K=\operatorname{Frac}(\cO)$. The preceding identification shows that the source of
\eqref{eq:bar-varpi-before-local-quotient}, after tensoring over $\cO$
with $K$, is the source of \cite[Theorem~1.2]{shih-wang}. That theorem
therefore gives
\[
 \bar\varpi_\theta\otimes_{\cO}K\colon
 \bigl(H^-_{\theta,\gm}(1)/\gp_fH^-_{\theta,\gm}(1)\bigr)
 \otimes_{\cO}K
 \xrightarrow{\;\sim\;}
 \bigl(S_{\theta,\gm_\Lambda}/\gp_1S_{\theta,\gm_\Lambda}\bigr)
 \otimes_{\cO}K.
\]
The Gorenstein sign basis in \eqref{eq:isom_H_h}
identifies the source of $\bar\varpi_\theta$ with $\cO$; hence its
kernel is zero.

The kernel of $q_\chi$ is finite, while $S'_\chi(1)$ is free of rank
one over $\cO$ by \cite[Proposition~3.2]{shih-wang}. Therefore
$q_\chi\otimes_{\cO}K$ is an isomorphism. It follows that
$\bar\varpi'_\theta\otimes_{\cO}K$ is an isomorphism as well. Since
the source of $\bar\varpi'_\theta$ is free
of rank one over $\cO$, the map $\bar\varpi'_\theta$ is injective.
\end{proof}

\begin{remark}\label{rem:integrality-exponent}
Let $e_f^-$ be the image modulo $\gp_f$ of the basis of
$H^-_{\theta,\gm}(1)$ fixed by $\iota^-_\theta$. Choose a compatible
system of $p$-power roots of unity and a local-invariant basis
$\beta_\chi$ of $S'_\chi(1)$ as in
\cite[Proposition~3.2]{shih-wang}. We normalize this basis on the
$\chi$-component using the relative cyclotomic trace in
\cite[proof of Theorem~4.2]{shih-wang}; using the absolute local
trace instead changes the scalar by the unramified degree
$[\Q_p(\mu_N):\Q_p]$, a unit since $p\nmid\phi(N)$. We choose its
sign to match the cup-product ordering in \eqref{eq:local_pairing_1}.
Write
\[
 \bar\varpi'_\theta(e_f^-)=c_f\beta_\chi,
 \qquad c_f\in\cO\setminus\{0\},
\]
and put $t:=v_\pi(c_f)$. Equivalently, the image of
$\bar\varpi'_\theta$ is $\pi^t\cO\beta_\chi$ up to an element of
$\cO^\times$. We will prove in
Corollary~\ref{cor:primitive-varpi} that $t=0$.
\end{remark}

Let $\cL_\theta^+$ be the component of $\cL_\theta$ defined using the
plus Manin symbols $[Na:1]_{r,\theta}^+$. The superscript refers to
these Manin symbols; under the Tate-twist convention of
\eqref{eq:coho_relative_homo}, this component lies in
$H_\theta^-\lsem\Z_p^\times\rsem$. By
\cite[Proposition~5.2.12]{FK12}, $\varpi_\theta$ sends it to
\begin{equation}\label{eq:image_MKpadicL_under_varpi}
 \cC_\theta
 :=\left(\sum_{a\in(\Z/p^r\Z)^\times}
 (1-\zeta_{p^r}^a,1-\zeta_{Np^r})_\theta[a]\right)_{r\geq1}
 \in S_\theta\lsem\Z_p^\times\rsem.
\end{equation}
Let $\cC_{\theta,\gm_\Lambda}$ be the image of $\cC_\theta$ in
$S_{\theta,\gm_\Lambda}\lsem\Z_p^\times\rsem$, and set
\[
 \cC'_{\theta}:=(q_{\chi}\otimes1)
 \bigl(\cC_{\theta,\gm_\Lambda}\bmod\gp_1\bigr)
 \in S'_{\chi}(1)\lsem\Z_p^{\times}\rsem.
\]
Since $\varpi_{\theta,\gm}$ factors through the Eisenstein quotient,
$\cC'_{\theta}$ is the image of the weight-one specialization of
$\cL_{\theta}^+$ under the coefficientwise extension of
$\bar\varpi'_{\theta}$. We use the basis $\beta_\chi$ from
Remark~\ref{rem:integrality-exponent} to identify
$S'_\chi(1)/p^r$ with $\cO/p^r\cO$, and write
$(v,w)'_{r,\theta}$ for the coordinate of the image of the
$\theta$-component of the cup product $(v,w)_r$.

Let $\rho$ be a finite-order character of conductor $p^r$. If
$\operatorname{ev}_\rho(\cC'_\theta)_r$ denotes the reduction modulo
$p^r$ of its evaluation, then the standard evaluation map in
\eqref{eq:branch-evaluation-convention} gives
\begin{equation}\label{eq:image_twovar_padi_L_under_varpi}
 \operatorname{ev}_\rho(\cC'_\theta)_r
 =\sum_{a\in(\Z/p^r\Z)^\times}\rho(a)
 (1-\zeta_{p^r}^a,1-\zeta_{Np^r})'_{r,\theta}.
\end{equation}
For the local calculation, let $\cO_\psi$ be the ring of integers of a
finite extension of $\Q_p$ containing $\cO$, the values of $\psi$, and
the Gauss sums occurring below; all modules are implicitly extended to
$\cO_\psi$. We put $\bar\psi:=\psi^{-1}$. Also, $p^{-r}$ in an
exponent of $\zeta_N$ denotes the inverse of $p^r$ in
$(\Z/N\Z)^\times$.

We will also use the Gauss-sum unit calculation
\begin{equation}\label{eq:gauss-sum-unit}
 \tau(\chi^{-1})\in\cO^\times.
\end{equation}
This follows from $\chi(p)=1$ and $p\nmid N$; see the calculation
recalled in \cite[Introduction]{shih-wang}.

\begin{prop}\label{local_cup_prod}
Let $\psi$ be a nontrivial even character of $\Z_p^{\times}$ of conductor $p^r$, where $r\geq 1$. Then
\[
\begin{split}
&\sum_{a\in (\Z/p^r\Z)^{\times}} \bar\psi(a)
 (1-\zeta_{p^r}^a,1-\zeta_N^{p^{-r}}\zeta_{p^r})'_{r,\theta}\\
 &\hspace{35mm}\equiv-\frac{\tau(\chi^{-1})}{\phi(N)}
 L_p(1,\psi)L_p(0,\chi\omega_p\bar\psi)
 \pmod{p^r\cO_\psi}.
\end{split}
\]
\end{prop}

\begin{proof}
We follow the explicit reciprocity calculation in the proof of
\cite[Theorem~4.2]{shih-wang}. In the multiplicative tensor product of
local units with $\cO_\psi$, set
\[
 \varepsilon_{\bar\psi,r}
 :=\prod_{a\in(\Z/p^r\Z)^{\times}}
       (1-\zeta_{p^r}^a)^{\bar\psi(a)},
 \qquad
\ell_{\psi,r}:=\log_p(\varepsilon_{\bar\psi,r})
 =\sum_a\bar\psi(a)\log_p(1-\zeta_{p^r}^a).
\]
Here the exponents in the product are interpreted through the
$\cO_\psi$-module structure on the completed tensor product; they are
not ordinary integral exponents. The element
$\varepsilon_{\bar\psi,r}$ is a local unit: the valuation of every factor
$1-\zeta_{p^r}^a$ is the same, whereas
$\sum_a\bar\psi(a)=0$ because $\psi$ is nontrivial. Moreover, for
$h\in(\Z/p^r\Z)^{\times}$,
\begin{equation}\label{eq:psi-log-eigenrelation}
 \sigma_h(\ell_{\psi,r})=\psi(h)\ell_{\psi,r}.
\end{equation}

Apply the $\theta=\chi\omega_p$ idempotent and reindex the
$N$-component as follows: if $g_N\in(\Z/N\Z)^\times$, put
$G_N=p^{-r}g_N$. Using $\chi(p)=1$, the left-hand side of the
proposition becomes
\begin{equation}\label{eq:local_pairing_1}
 \frac{1}{\phi(Np)}
 \sum_{G_N,g_p}\chi(G_N^{-1})\omega_p(g_p^{-1})
 \psi(\omega_p(g_p))
 (\varepsilon_{\bar\psi,r},
  1-\zeta_N^{G_N}\zeta_{p^r}^{\omega_p(g_p)})_r,
\end{equation}
where $G_N$ and $g_p$ range over $(\Z/N\Z)^{\times}$ and
$(\Z/p\Z)^{\times}$, respectively; the image in $S'_{\chi}(1)$ is
understood.

Put $K_0=\Q_p(\mu_N)$ and $K_r=K_0(\mu_{p^r})$. For the second
entry in \eqref{eq:local_pairing_1}, we use the Coleman power series
\[
 g_{G_N,g_p}(Z)=1-\zeta_N^{G_N}(1+Z)^{\omega_p(g_p)}.
\]
For a Coleman power series $g$, write
$\delta g:=(1+Z)g'(Z)/g(Z)$. Then
\[
 \delta g_{G_N,g_p}(\zeta_{p^r}-1)
 =\omega_p(g_p)
 \frac{\zeta_N^{G_N}\zeta_{p^r}^{\omega_p(g_p)}}
 {\zeta_N^{G_N}\zeta_{p^r}^{\omega_p(g_p)}-1}.
\]
To see that the second entry is a universal norm, above level $r$ replace
$\zeta_N^{G_N}$ at level $m$ by
$\zeta_N^{p^{r-m}G_N}$, where the exponent is taken modulo $N$.
The product over the $p$ lifts of $\zeta_{p^{m-1}}$ gives the norm
relation, and the Frobenius shift on coefficients gives the displayed
Coleman series at level $r$. The first entry, after passage to the
$p$-completion, belongs to the completed tensor product of principal
units. Bilinearity therefore allows us to apply the local explicit
reciprocity law
\cite[Theorem~8.18]{iwasawa86} (equivalently,
\cite[Theorem~I.4.2]{deS}). With the cup-product ordering used here,
this identifies \eqref{eq:local_pairing_1} with
\[
 -\frac{1}{p^r\phi(Np)}\sum_{G_N,g_p}
 \chi(G_N^{-1})\omega_p(g_p^{-1})\psi(\omega_p(g_p))
 \Tr_{K_r/K_0}\!\left(
 \ell_{\psi,r}\,\omega_p(g_p)
 \frac{\zeta_N^{G_N}\zeta_{p^r}^{\omega_p(g_p)}}
 {\zeta_N^{G_N}\zeta_{p^r}^{\omega_p(g_p)}-1}
 \right).
\]
Evaluating the trace using \eqref{eq:psi-log-eigenrelation} gives
\begin{equation}\label{eq:comp_tr}
 -\frac{1}{p^r\phi(N)}\ell_{\psi,r}
 \sum_{G_N,G_{p^r}}\chi(G_N^{-1})\psi(G_{p^r})
 \frac{\zeta_N^{G_N}\zeta_{p^r}^{G_{p^r}}}
 {\zeta_N^{G_N}\zeta_{p^r}^{G_{p^r}}-1},
\end{equation}
where $G_{p^r}$ ranges over $(\Z/p^r\Z)^{\times}$. Here we used
$(p-1)/\phi(Np)=1/\phi(N)$.

For $z=\zeta_N^{G_N}\zeta_{p^r}^{G_{p^r}}$,
\cite[Lemma~4.1(2)]{shih-wang} gives
\[
 \frac{z}{z-1}
 =-\sum_{\alpha\bmod Np^r}
 \zeta_{\alpha\,(Np^r)}(0)z^\alpha,
\]
where $\zeta_{\alpha\,(Np^r)}$ is the corresponding partial Riemann
zeta function. For each $\alpha$, the two Gauss sums give
\[
 \sum_{G_N,G_{p^r}}\chi(G_N^{-1})\psi(G_{p^r})
 (\zeta_N^{G_N}\zeta_{p^r}^{G_{p^r}})^\alpha
 =
 \begin{cases}
 \tau(\chi^{-1})\tau(\psi)(\chi\bar\psi)(\alpha),
     &(\alpha,Np)=1,\\
 0,&(\alpha,Np)>1.
 \end{cases}
\]
Consequently, \eqref{eq:comp_tr} is
\begin{equation}\label{eq:local-pairing-before-L-values}
 \frac{\tau(\chi^{-1})\tau(\psi)}{p^r\phi(N)}
 \ell_{\psi,r}L(0,\chi\bar\psi).
\end{equation}

Since $\psi$ is primitive of conductor $p^r$, it is ramified at $p$.
Thus \cite[Theorem~5.18]{washington-GTM83} and the interpolation
formula give, respectively,
\[
 L_p(1,\psi)=-\frac{\tau(\psi)}{p^r}\ell_{\psi,r},
 \qquad
 L_p(0,\chi\omega_p\bar\psi)=L(0,\chi\bar\psi).
\]
Substituting these identities into
\eqref{eq:local-pairing-before-L-values} proves the claim.
\end{proof}

We now return to the standard roots used in
\eqref{eq:image_MKpadicL_under_varpi} and to the evaluation convention
\eqref{eq:branch-evaluation-convention}.

\begin{cor}\label{local_cup_prod_2}
Let $\rho$ be a nontrivial even character of $\Z_p^\times$ of
conductor $p^r$, and let $\cO_\rho$ be the ring of integers of a finite
extension containing $\cO$ and the values of $\rho$. Then
\begin{equation}\label{eq:standard-root-finite-congruence}
 \operatorname{ev}_\rho(\cC'_\theta)_r
 \equiv
 -\frac{\tau(\chi^{-1})N^{-1}}{\phi(N)}
 \rho(N)^{-1}L_p(1,\rho^{-1})L_p(0,\chi\omega_p\rho)
 \pmod{p^r\cO_\rho}.
\end{equation}
\end{cor}

\begin{proof}
Put
\[
 z_r:=\zeta_N^{p^{-r}}\zeta_{p^r}
     =\zeta_{Np^r}^{d_r},
\]
where $d_r\equiv1\pmod N$ and $d_r\equiv N\pmod{p^r}$. Let
\[
 P_\rho:=\sum_a\rho(a)(1-\zeta_{p^r}^a,1-z_r)'_{r,\theta},
 \qquad
 S_\rho:=\sum_a\rho(a)
 (1-\zeta_{p^r}^a,1-\zeta_{Np^r})'_{r,\theta}.
\]
Reindexing the first sum gives
\[
 \sigma_{d_r^{-1}}P_\rho=\rho(N)S_\rho.
\]
On $S'_\chi(1)/p^r$, the tame component of $d_r^{-1}$ is
trivial and its cyclotomic component acts by $N^{-1}$ modulo $p^r$
(the Tate twist); see \cite[proof of Proposition~3.2]{shih-wang}.
This uses the full cyclotomic character, not just its Teichm\"uller part.
Consequently,
\[
 S_\rho=N^{-1}\rho(N)^{-1}P_\rho\quad\bmod p^r.
\]
Apply Proposition~\ref{local_cup_prod} with
$\psi=\rho^{-1}$ and use
\eqref{eq:image_twovar_padi_L_under_varpi}.
\end{proof}

\begin{lemma}[Finite-level reconstruction]
\label{lem:finite-level-reconstruction}
Let $A$ be the ring of integers of a finite extension of $\Q_p$, and
let $F(X)\in A\lsem X\rsem$. Suppose that, for every sufficiently large
$r$, there is a primitive character $\psi_r$ of conductor $p^r$ such
that
\[
 v_p\!\left(F\bigl(\psi_r(u)-1\bigr)\right)\geq r,
\]
where the value is taken in $\bar\Q_p$ and $v_p(p)=1$.
Then $F=0$.
\end{lemma}

\begin{proof}
Suppose that $F\neq0$ and write its Weierstrass factorization as
$F=\pi_A^\mu P U$, where $\pi_A$ is a uniformizer of $A$,
$\mu\in\Z_{\geq0}$, $P$ is distinguished of degree $d$, and $U$ is a
unit. For a primitive $\psi_r$, the valuation of
$x_r:=\psi_r(u)-1$ tends to zero as $r\to\infty$. For all sufficiently
large $r$, the leading term $x_r^d$ is the unique term of smallest
valuation in $P(x_r)$; hence
\[
 v_p(F(x_r))=v_p(\pi_A^\mu)+d\,v_p(x_r).
\]
The right-hand side is bounded, contradicting
$v_p(F(x_r))\geq r$.
\end{proof}

Write
\[
 \langle N\rangle:=\omega_p(N)^{-1}N=u^{s_N}
 \quad(s_N\in\Z_p),
 \qquad
 \mathcal U_N(X):=(1+X)^{s_N}\in\Lambda_\cO^\times.
\]
For even $j$, let $\cC'_{\theta,j}(X)$ denote the
$\beta_\chi$-coordinate of the $\omega_p^j$-branch of
$\cC'_\theta$.

\begin{cor}\label{cor:positive-even-cup-branches}
If $j>0$ is even, then
\begin{equation}\label{eq:positive-even-cup-branch}
\begin{split}
 \cC'_{\theta,j}(X)
  ={}&-\frac{\tau(\chi^{-1})}{\phi(N)}
  N^{-1}\omega_p(N)^j\mathcal U_N(X)\\
 &\quad\cdot
 G_{\omega_p^j}(X)
 G_{\chi\omega_p^{1-j}}\bigl(u(1+X)^{-1}-1\bigr).
\end{split}
\end{equation}
\end{cor}

\begin{proof}
Let $\psi_p$ run through the primitive finite-order characters of
$1+p\Z_p$, put $\rho=\omega_p^{-j}\psi_p^{-1}$, and set
$X_{\psi_p}=\psi_p(u)-1$. Then
\[
 \rho(N)^{-1}
 =\omega_p(N)^j\mathcal U_N(X_{\psi_p}).
\]
Equations \eqref{eq:kubota_leopoldt_power_series} and
\eqref{eq:KL_padic_L_typeW} give
\[
 L_p(1,\rho^{-1})
 =G_{\omega_p^j}(X_{\psi_p})
\]
and
\[
 L_p(0,\chi\omega_p\rho)
 =G_{\chi\omega_p^{1-j}}
   \bigl(u(1+X_{\psi_p})^{-1}-1\bigr).
\]
Thus Corollary~\ref{local_cup_prod_2} says that the difference between
the two sides of \eqref{eq:positive-even-cup-branch}, evaluated at
$X_{\psi_p}$, is divisible by $p^r$ when $\psi_p$ has conductor $p^r$.
Lemma~\ref{lem:finite-level-reconstruction} proves the identity.
\end{proof}

\begin{cor}\label{cor:primitive-varpi}
The integer $t$ in Remark~\ref{rem:integrality-exponent} is zero. In particular,
$\bar\varpi'_{\theta}$ maps isomorphically onto $S'_{\chi}(1)$.
\end{cor}

\begin{proof}
Apply Corollary~\ref{cor:positive-even-cup-branches} with $j=2$.
The Gauss sum $\tau(\chi^{-1})$ belongs to $\cO^\times$ by the unit
calculation recalled in \cite[Introduction]{shih-wang}, and
$p\nmid\phi(N)$. Both characters $\omega_p^2$ and
$\chi\omega_p^{-1}$ are nontrivial and even. The theorem of
Ferrero--Washington shows that their Kubota--Leopoldt power series have
$\mu$-invariant zero. The factor $\mathcal U_N$ is a unit, and the
substitution in the second factor of
\eqref{eq:positive-even-cup-branch} preserves the $\mu$-invariant.
Therefore
\[
\mu\bigl(\cC'_{\theta,2}\bigr)=0.
\]

On the other hand, the image of the coefficient map
$\bar\varpi'_{\theta}$ is $\pi^t\cO\beta_\chi$. Consequently, every branch of
$\cC'_{\theta}$, which is the image of the weight-one specialization
of $\cL_{\theta}^+$ under the coefficientwise extension of
$\bar\varpi'_{\theta}$, belongs to
$\pi^t\Lambda_{\cO}$. The branch considered above would therefore
have $\mu$-invariant at least $t$. Since its $\mu$-invariant is zero,
we obtain $t=0$. The last assertion follows from Lemma~\ref{cup_prod_modulo_pf}.
\end{proof}

We can now pass from the local cup product to the even branches of the
$p$-adic $L$-function.

\begin{thm}\label{p_adic_L_wt_one_plus}
Assume that the cuspidal Hecke algebra $\gh_{\theta,\gm}$ is Gorenstein.
Put
\[
 \mathcal V_0(X):=\mathcal U_N(X)\mathscr U_{\mathbbm{1}}(X).
\]
Then $\mathcal V_0(X)\in\Lambda_{\cO}^{\times}$. For every even
integer $0\leq j\leq p-2$, we have
\begin{equation}\label{eq:i_branch}
\cL(f,\omega_p^j)(X)\sim_{\cO}
\begin{cases}
\mathcal V_0(X)
\dfrac{G_{\chi\omega_p}(u(1+X)^{-1}-1)}{X},&j=0,\\[8pt]
\mathcal U_N(X)
G_{\omega_p^j}(X)
G_{\chi\omega_p^{1-j}}(u(1+X)^{-1}-1),&j>0.
\end{cases}
\end{equation}
The quotient in the case $j=0$ belongs to $\Lambda_{\cO}$. Moreover,
$\mu^{\an}(f,j)=0$, and
\begin{equation}\label{eq:even-analytic-lambda}
\lambda^{\an}(f,j)=
\begin{cases}
\lambda\bigl(G_{\chi\omega_p}\bigr)-1,&j=0,\\[4pt]
\lambda\bigl(G_{\omega_p^j}\bigr)
+\lambda\bigl(G_{\chi\omega_p^{1-j}}\bigr),&j>0.
\end{cases}
\end{equation}
\end{thm}

\begin{proof}
By Corollary~\ref{cor:primitive-varpi}, the coefficient map
$\bar\varpi'_{\theta}$ is an isomorphism. Relative to the fixed even
sign basis and the local basis $\beta_\chi$, it differs from the
identity by an element of $\cO^\times$. Thus
Corollary~\ref{cor:positive-even-cup-branches} gives the formula for
$j>0$.

Both characters $\omega_p^j$ and $\chi\omega_p^{1-j}$ are
nontrivial and even. The theorem of Ferrero--Washington
\cite{ferrero-washington} therefore shows that both factors have
$\mu$-invariant zero. The factor $\mathcal U_N$ is a unit, and the
substitution in the second factor of \eqref{eq:i_branch} preserves the
Iwasawa invariants. This proves the assertions for $j>0$.

Now let $j=0$, and set
\[
 B(X):=u(1+X)^{-1}-1.
\]
The trivial-character factor satisfies
\[
 G_{\mathbbm{1}}(X)=\frac{\mathscr U_{\mathbbm{1}}(X)}{X}.
\]
On the other hand,
\begin{equation}\label{eq:even-trivial-zero}
 G_{\chi\omega_p}(u-1)=L_p(0,\chi\omega_p)
 =(1-\chi(p))L(0,\chi)=0.
\end{equation}
Hence $X$ divides $G_{\chi\omega_p}(B(X))$, so
\begin{equation}\label{eq:even-trivial-zero-product}
 \mathcal U_N(X)G_{\mathbbm{1}}(X)G_{\chi\omega_p}(B(X))
 =\mathcal V_0(X)\frac{G_{\chi\omega_p}(B(X))}{X}\in\Lambda_\cO.
\end{equation}
Apply Corollary~\ref{local_cup_prod_2} to primitive characters of
$1+p\Z_p$. After multiplication by the fixed unit relating the two
bases, the difference between $\cL(f,\mathbbm{1})$ and the last power
series has values divisible by $p^r$ at every primitive character of
conductor $p^r$. Lemma~\ref{lem:finite-level-reconstruction} proves the
first case of \eqref{eq:i_branch}. Finally, the theorem of
Ferrero--Washington gives
$\mu(G_{\chi\omega_p})=0$; division by the distinguished linear
polynomial $X$ lowers the $\lambda$-invariant by one and leaves the
$\mu$-invariant unchanged. Formula \eqref{eq:even-analytic-lambda}
follows.
\end{proof}

\subsection{The Fukaya--Kato formula on the odd branches}\label{sec:odd_analytic}

Fukaya and Kato compute the plus Mazur--Kitagawa element modulo the
Eisenstein ideal in \cite[Theorem~8.2.2(1)]{FK12}. The character
$\theta=\chi\omega_p$ falls outside their Assumption~3. The adjustment
needed here is to identify the plus Eisenstein quotient for this
character (Lemma~\ref{lem:exceptional-canonical-plus}). With this
identification, their calculation applies and specializes at $f$ to
give the odd branches. We use their formula at the level of the
family, since their classical specialization
\cite[Proposition~8.2.4]{FK12} requires weight at least two.

\begin{lemma}[Exceptional canonical plus quotient]
\label{lem:exceptional-canonical-plus}
Recall that
$\gm_\Lambda=\gm\cap\Lambda_{\cO}^{\mathrm{wt}}$ and
$\mathscr A=(\Lambda_{\cO}^{\mathrm{wt}})_{\gm_\Lambda}$. Put
\[
 R:=\gh_{\theta,\gm},
 \qquad
 I:=I_{\theta}R.
\]
Let $\xi_{\theta}\in\Lambda_{\cO}^{\mathrm{wt}}$ be the Eisenstein
power series satisfying
\[
 R/I\simeq\mathscr A/(\xi_{\theta});
\]
this is the localization of \cite[(3.10)]{shih-wang}. Put
\[
 z_{\theta}:=\xi_{\theta}
 \{0,\infty\}_{\theta,\mathrm{DM}}\in H_{\theta,\gm}^{-}.
\]
If $R$ is Gorenstein, Ohta's pairing induces an isomorphism
\begin{equation}\label{eq:exceptional-canonical-plus}
 \begin{split}
 \Phi^{+}\colon
 H_{\theta,\gm}^{+}/I H_{\theta,\gm}^{+}
 &\xrightarrow{\;\sim\;} \mathscr A/(\xi_{\theta}),\\
 x&\longmapsto
 (x,z_{\theta})_{\Lambda_{\cO}^{\mathrm{wt}}}
 \pmod{\xi_{\theta}}.
 \end{split}
\end{equation}
\end{lemma}

\begin{proof}
Let
\[
 F:=\widetilde H_{\theta,\mathrm{DM},\gm},
 \qquad
 e:=\{0,\infty\}_{\theta,\mathrm{DM}}.
\]
Here $F$ is the localized Drinfeld--Manin modification. By
\cite[Proposition~3.3]{shih-wang},
\[
 F/H_{\theta,\gm}\simeq\mathscr A/(\xi_{\theta}),
\]
the class of $e$ generates this quotient, and $e$ is part of an
$\mathscr A$-basis of $F$. Extend it to a basis
$e,e_2,\ldots,e_d$. After replacing $e_i$ by
$e_i-a_i e$ for suitable $a_i\in\mathscr A$, we may assume that
$e_2,\ldots,e_d$ belong to $H_{\theta,\gm}$. It follows that
\[
 H_{\theta,\gm}
 =
 \mathscr A(\xi_{\theta}e)\oplus
 \bigoplus_{i=2}^{d}\mathscr A e_i.
\]
Thus $z_{\theta}=\xi_{\theta}e$ is primitive in
$H_{\theta,\gm}$.

Ohta's pairing is perfect. Consequently, pairing against
$z_{\theta}$ is surjective onto $\mathscr A$. The same calculation as in
\cite[\S~6.3.8]{FK12} shows that it descends modulo $I$ to a
surjection
\[
 H_{\theta,\gm}/I H_{\theta,\gm}
 \longrightarrow\mathscr A/(\xi_{\theta}).
\]
The required integrality of $z_{\theta}$ follows from
\cite[Corollary~3.4]{shih-wang}.

It remains to identify the eigenspace through which this surjection
factors.
Let $c$ be complex conjugation. The equivariance formula for Ohta's
pairing \cite[(5.1)]{shih-wang} has multiplier $-1$ at $c$, because
$\theta$ is even and $\kappa(c)=-1$. Since
$z_{\theta}\in H_{\theta,\gm}^{-}$, this gives
\[
 (H_{\theta,\gm}^{-},z_{\theta})_{\Lambda_{\cO}^{\mathrm{wt}}}=0;
\]
here we use that $2$ is invertible because $p\geq5$. Hence the preceding
surjection factors through the plus quotient in
\eqref{eq:exceptional-canonical-plus}. Finally, the Gorenstein
hypothesis makes $H_{\theta,\gm}^{+}$ a free $R$-module of rank one.
Both sides of \eqref{eq:exceptional-canonical-plus} are therefore free
of rank one over the local ring
$R/I\simeq\mathscr A/(\xi_{\theta})$, so the surjection is an
isomorphism.
\end{proof}

\begin{prop}[Fukaya--Kato specialization formula for the exceptional character]
\label{prop:exceptional-fk-specialization}
Let $\,{}^{0}\cL_{\theta,\gm}^{+}$ be the image of the plus
Mazur--Kitagawa element under the coefficient twist
\[
 x\otimes[a]\longmapsto a^{-1}x\otimes[a]
\]
of \cite[\S~8.2.1]{FK12}. Under the coefficientwise extension of
\eqref{eq:exceptional-canonical-plus}, one has
\begin{equation}\label{eq:exceptional-fk-family}
 \Phi^{+}\bigl({}^{0}\cL_{\theta,\gm}^{+}\bmod I\bigr)
 =
 2\sigma_{-1}A_0\widetilde{\xi}
 \quad\text{in}\quad
 \bigl(\mathscr A/(\xi_{\theta})\bigr)
 \lsem\Z_p^{\times}\rsem,
\end{equation}
where $\sigma_{-1}$ is the group element attached to $-1$ and $A_0$
and $\widetilde\xi$ are the elements appearing in
\cite[Theorems~8.1.2 and~8.2.2]{FK12}. More precisely,
$\widetilde\xi$ is obtained from the element $B_0$ of
\cite[Proposition~4.3.6]{FK12} by the group-ring involution
$[a]\mapsto[a^{-1}]$ and the $\theta$-projection prescribed in the
proof of \cite[Theorem~8.2.2(1)]{FK12}.
\end{prop}

\begin{proof}
This is the exceptional analogue of
\cite[Theorem~8.2.2(1)]{FK12}. We indicate why its proof remains
integral here. In that proof, the input from \cite[\S~6.3.18]{FK12} is
the canonical identification of the plus Eisenstein quotient with the
quotient of the localized weight algebra (denoted by
$\Lambda_\theta/(\xi_\theta)$ in that reference). In the present
exceptional setting, Lemma~\ref{lem:exceptional-canonical-plus}
provides exactly this identification with $\mathscr A/(\xi_\theta)$.

For completeness, let $e=\{0,\infty\}_{\theta,\mathrm{DM}}$ and
$z_{\theta}=\xi_{\theta}e$. Apply \cite[Theorem~4.4.3]{FK12} to
Ohta's pairing and then follow the proof of
\cite[Theorem~8.2.2(1)]{FK12}. After taking the coefficient of $q$, apply
the group-ring involution $[a]\mapsto[a^{-1}]$ used there. The
same coefficient calculation, with $A$ and $B$ as in
\cite[Proposition~4.3.6]{FK12}, gives, modulo $\xi_\theta$, the element
\[
 2\sigma_{-1}A_0\widetilde\xi
 \in\bigl(\mathscr A/(\xi_\theta)\bigr)
 \lsem\Z_p^\times\rsem.
\]
This proves \eqref{eq:exceptional-fk-family}. The exceptional
Drinfeld--Manin input is supplied by
\cite[Proposition~3.3]{shih-wang}, while the integrality of
$z_{\theta}$ is supplied by \cite[Corollary~3.4]{shih-wang}. Thus the
argument does not require \cite[Assumption~3]{FK12}.
\end{proof}

\begin{thm}[Odd-branch factorization, after Fukaya--Kato]\label{p_adic_L_wt_one_minus}
Let $1\leq j\leq p-2$ be odd. Then
\begin{equation}\label{eq:odd-analytic-factorization}
\cL(f,\omega_p^j)(X)
\sim_{\cO}
G_{\chi^{-1}\omega_p^j}(X)\,
G_{\omega_p^{1-j}}
\bigl(u(1+X)^{-1}-1\bigr).
\end{equation}

If $j\neq1$, both factors on the right-hand side of
\eqref{eq:odd-analytic-factorization} are elements of $\Lambda$, and
\begin{equation}\label{eq:odd-analytic-mu}
\mu^{\an}(f,j)=0
\end{equation}
and
\begin{equation}\label{eq:odd-analytic-lambda}
\lambda^{\an}(f,j)
=
\lambda\bigl(G_{\chi^{-1}\omega_p^j}\bigr)
+
\lambda\bigl(G_{\omega_p^{1-j}}\bigr).
\end{equation}

If $j=1$, the right-hand side of \eqref{eq:odd-analytic-factorization} is initially
interpreted in the total quotient ring of $\Lambda$. The simple pole
of the second factor is cancelled by the trivial zero of the first
factor. More precisely, if
\[
Y_1(X):=u(1+X)^{-1}-1,
\qquad
\ell_1(X):=1+X-u,
\]
and
\[
 \mathcal V_1(X):=-(1+X)
 \mathscr U_{\mathbbm{1}}\bigl(Y_1(X)\bigr)
 \in\Lambda_{\cO}^{\times},
\]
then
\begin{equation}\label{eq:odd-analytic-exceptional}
\cL(f,\omega_p)(X)
\sim_{\cO}
\mathcal V_1(X)\frac{G_{\chi^{-1}\omega_p}(X)}{\ell_1(X)}.
\end{equation}
Consequently,
\begin{equation}\label{eq:odd-analytic-mu-exceptional}
\mu^{\an}(f,1)=0
\end{equation}
and
\begin{equation}\label{eq:odd-analytic-lambda-exceptional}
\lambda^{\an}(f,1)
=
\lambda\bigl(G_{\chi^{-1}\omega_p}\bigr)-1.
\end{equation}
\end{thm}

\begin{proof}
Let
\[
\psi' : 1+p\Z_p
\longrightarrow
\bar\Q_p^\times
\]
be a finite-order character, and put
\[
 \overline{\psi'}:=(\psi')^{-1},
 \qquad
 \psi=\omega_p^j\psi',
 \qquad
 \alpha:=\psi^{-1}=\omega_p^{-j}\overline{\psi'}.
\]

We first reconcile the dual-Hecke notation with that of Fukaya--Kato.
In \cite[Sections~4.5.3--4.5.5 and~7.1.11]{FK12}, the character
$\varepsilon$ is the nebentype of the original eigenform $g$, whereas
the cohomology with dual Hecke action is projected to the packet of
$g^*$. The arithmetic weight map is normalized by
\[
 [a]_{\mathrm{wt}}\longmapsto a^{2-k}\varepsilon^{-1}(a),
\]
so its torsion character is
\begin{equation}\label{eq:fk-character-convention}
 \theta=\omega_p^{2-k}\varepsilon^{-1}.
\end{equation}
At our weight-one cohomological point, the dual packet is represented
by $g^*=f=E_1(\chi,\mathbbm{1})$ (with the ordinary
$p$-stabilization understood). Thus the original character in the
notation of \textit{loc.~cit.} is
\begin{equation}\label{eq:weight-one-epsilon}
 \varepsilon=\chi^{-1},
\end{equation}
which is equivalent to $\theta=\chi\omega_p$ in
\eqref{eq:fk-character-convention}. As in
\eqref{eq:def_padic_Lfcn_of_f}, we denote the
Mazur--Kitagawa specialization at this cohomological point by
$\cL(f)$.

By the branch convention \eqref{eq:branch-evaluation-convention}, the
left-hand side of the desired branch identity is
\[
 \operatorname{ev}_\alpha\bigl(\cL(f)\bigr).
\]
Specializing the coefficient twist that defines ${}^{0}\cL$ gives the
exact relation
\begin{equation}\label{eq:zero-L-kappa-shift}
 \operatorname{ev}_\alpha\bigl(\cL(f)\bigr)
 =\operatorname{ev}_{\kappa\alpha}\bigl({}^{0}\cL(f)\bigr),
\end{equation}
because the twist multiplies the coefficient of $[a]$ by $a^{-1}$.

Since $j$ is odd,
\[
 \alpha(-1)=-1,
 \qquad
 (\kappa\alpha)(-1)=1,
\]
so $\kappa\alpha$ selects the plus cohomological component occurring in
\eqref{eq:exceptional-fk-family}. It also evaluates
$\sigma_{-1}$ to $1$.

Since $I_{\theta}\subset\gp_f$, we may now specialize the
family-level identity \eqref{eq:exceptional-fk-family} at $\gp_f$ and
then evaluate it at $\kappa\alpha$. The definitions in
\cite[Sections~4.1.1 and~8.1.2]{FK12} first give values imprimitive at
$p$ and $Np$, respectively. The character $\alpha$ is ramified at
$p$, while $\varepsilon\alpha^{-1}$ is ramified at every prime
dividing $Np$. All omitted Euler factors are therefore $1$, and hence
\[
 \operatorname{ev}_{\kappa\alpha}(A_0)=L(0,\alpha),
 \qquad
 \operatorname{ev}_{\kappa\alpha}(\widetilde\xi\bmod\gp_f)
 =L_p(1,\varepsilon\alpha^{-1}).
\]
Together with \eqref{eq:weight-one-epsilon}, this gives
\begin{equation}\label{eq:odd-specialization}
 \cL(f,\omega_p^j)\bigl(\psi'(u)-1\bigr)
\sim_{\cO}
L(0,\psi^{-1})\,
L_p(1,\chi^{-1}\psi).
\end{equation}
The unit in \eqref{eq:odd-specialization} is independent of $\psi'$
because it comes from the change of basis relating the generator determined by
\eqref{eq:exceptional-canonical-plus} to the plus Gorenstein basis fixed
in \eqref{eq:isom_H_h}. Hence
\eqref{eq:odd-specialization} comes from a single identity in the total
quotient ring of measures, up to one element of $\cO^\times$, rather
than from pointwise identities with varying constants.

The character $\psi^{-1}$ is ramified at $p$, so its Euler factor at
$p$ is $1$. We can therefore express the first factor as a
Kubota--Leopoldt value,
\[
L(0,\psi^{-1})
=
L_p(0,\omega_p\psi^{-1}).
\]
Since
\[
\omega_p\psi^{-1}
=
\omega_p^{1-j}\overline{\psi'},
\]
we have
\begin{equation}\label{eq:odd-first-factor}
L(0,\psi^{-1})
=
L_p
\bigl(
0,\omega_p^{1-j}\overline{\psi'}
\bigr).
\end{equation}
The second factor is
\begin{equation}\label{eq:odd-second-factor}
L_p
\bigl(
1,\chi^{-1}\omega_p^j\psi'
\bigr).
\end{equation}

Set
\[
X_{\psi'}:=\psi'(u)-1.
\]
Using \eqref{eq:kubota_leopoldt_power_series} and \eqref{eq:KL_padic_L_typeW}, equation
\eqref{eq:odd-second-factor} becomes
\begin{equation}\label{eq:odd-second-power-series}
L_p
\bigl(
1,\chi^{-1}\omega_p^j\psi'
\bigr)
=
G_{\chi^{-1}\omega_p^j}(X_{\psi'}),
\end{equation}
whereas \eqref{eq:odd-first-factor} becomes
\begin{equation}\label{eq:odd-first-power-series}
\begin{aligned}
L_p
\bigl(
0,\omega_p^{1-j}\overline{\psi'}
\bigr)
&=
G_{\omega_p^{1-j}}
\bigl(
u\overline{\psi'}(u)-1
\bigr)
\\
&=
G_{\omega_p^{1-j}}
\bigl(
u(1+X_{\psi'})^{-1}-1
\bigr).
\end{aligned}
\end{equation}

The preceding calculation supplies a single $c\in\cO^\times$,
independent of $\psi'$, such that the two sides of the desired identity agree after the
right-hand side is multiplied by $c$. If $j=1$, first clear the
trivial-character denominator $1+X-u$; no denominator is needed when
$j\neq1$. The resulting difference belongs to $\Lambda_\cO$ and
vanishes at $X=\psi'(u)-1$ for every primitive finite-order character
$\psi'$. Lemma~\ref{lem:finite-level-reconstruction} shows that this
difference is zero. Thus, in the total quotient ring,
\[
\cL(f,\omega_p^j)(X)
\sim_{\cO}
G_{\chi^{-1}\omega_p^j}(X)\,
G_{\omega_p^{1-j}}
\bigl(
u(1+X)^{-1}-1
\bigr).
\]

Suppose first that $j\neq1$. Since $j$ is odd,
\[
\chi^{-1}\omega_p^j
\]
is even, and it is nontrivial because its tame part is
$\chi^{-1}$. Moreover,
\[
\omega_p^{1-j}
\]
is even and nontrivial when $j\neq1$. Thus both factors in
\eqref{eq:odd-analytic-factorization} are ordinary Kubota--Leopoldt power series.

The theorem of Ferrero--Washington gives
\[
\mu\bigl(G_{\chi^{-1}\omega_p^j}\bigr)
=
\mu\bigl(G_{\omega_p^{1-j}}\bigr)
=
0.
\]
The substitution
\[
X\longmapsto u(1+X)^{-1}-1
\]
is an automorphism of $\cO\lsem X\rsem$ and therefore preserves the
$\mu$- and $\lambda$-invariants. This proves
\eqref{eq:odd-analytic-mu} and \eqref{eq:odd-analytic-lambda}.

We now assume $j=1$. Then
\[
\omega_p^{1-j}=1.
\]
The trivial-character Kubota--Leopoldt $p$-adic zeta function has a
simple pole at its augmentation point. Under the substitution
\[
Y_1(X)=u(1+X)^{-1}-1
=-\frac{1+X-u}{1+X}=-\frac{\ell_1(X)}{1+X},
\]
the pole occurs when $Y_1(X)=0$, namely when
\[
1+X=u.
\]
More precisely,
\begin{equation}\label{eq:trivial-character-pole}
G_{\mathbbm{1}}\bigl(Y_1(X)\bigr)
=-\frac{(1+X)\mathscr U_{\mathbbm{1}}(Y_1(X))}{\ell_1(X)}
=\frac{\mathcal V_1(X)}{\ell_1(X)}.
\end{equation}

On the other hand,
\[
G_{\chi^{-1}\omega_p}(u-1)
=
L_p(0,\chi^{-1}\omega_p)
=
\bigl(1-\chi^{-1}(p)\bigr)L(0,\chi^{-1})
=
0
\]
since $\chi(p)=1$. Hence
\begin{equation}\label{eq:odd-trivial-zero}
1+X-u
\mid
G_{\chi^{-1}\omega_p}(X).
\end{equation}
The zero in \eqref{eq:odd-trivial-zero} cancels the pole in
\eqref{eq:trivial-character-pole}, which gives \eqref{eq:odd-analytic-exceptional}.

Finally, Ferrero--Washington gives
\[
\mu\bigl(G_{\chi^{-1}\omega_p}\bigr)=0.
\]
Division by the distinguished linear polynomial $\ell_1(X)$
decreases the $\lambda$-invariant by one and does not alter the
$\mu$-invariant. This proves \eqref{eq:odd-analytic-mu-exceptional} and
\eqref{eq:odd-analytic-lambda-exceptional}.
\end{proof}

\subsection{Analytic invariants in the Hida family}\label{sec:analytic_hida}

Theorems~\ref{p_adic_L_wt_one_minus} and
\ref{p_adic_L_wt_one_plus} give the complete factorization of
$\cL(f)$, with the exceptional corrections occurring at $j=1$ and
$j=0$, respectively. We now record their common consequence for the
analytic Iwasawa invariants at the fixed tame level and transfer it to the
higher-weight arithmetic specializations of the localized Hida
component.

For every such specialization $g$, use the contragredient branch
convention \eqref{eq:branch-evaluation-convention} and set
\[
 \mu_N^{\an}(g,j):=\mu\bigl(\cL^{(N)}(g,\omega_p^j)\bigr),
 \qquad
 \lambda_N^{\an}(g,j):=\lambda\bigl(\cL^{(N)}(g,\omega_p^j)\bigr).
\]

\begin{cor}\label{ana_inv_(g,j)}
Let $g$ be a classical ordinary cuspidal eigenform of weight $k\geq2$
arising from an arithmetic point on a Hida component corresponding to
a minimal prime of $\gh_{\theta,\gm}$. Then
$\mu_N^{\an}(g,j)=0$ for every $0\leq j\leq p-2$. On the trivial
even branch,
\[
\lambda_N^{\an}(g,0)
 =\lambda\bigl(G_{\chi\omega_p}\bigr)-1.
\]
If $j>0$ is even, then
\[
\lambda_N^{\an}(g,j)
 =\lambda(G_{\omega_p^{j}})
 +\lambda(G_{\chi\omega_p^{1-j}}).
\]
If $j$ is odd and $j\neq1$, then
\[
\lambda_N^{\an}(g,j)
=\lambda(G_{\chi^{-1}\omega_p^j})
+\lambda(G_{\omega_p^{1-j}}),
\]
whereas
\[
\lambda_N^{\an}(g,1)=\lambda(G_{\chi^{-1}\omega_p})-1.
\]
If $g$ is new of exact tame level $N$, the same formulas hold for the
invariants $\mu^{\an}(g,j)$ and $\lambda^{\an}(g,j)$ of the primitive
$p$-adic $L$-function. When the primitive tame level is smaller than
$N$, removing the missing tame Euler factors gives the corresponding
standard correction to the $\lambda$-invariant.
\end{cor}

\begin{proof}
Theorem~\ref{p_adic_L_wt_one_minus} gives the formulas for the odd branches of $f$,
and Theorem~\ref{p_adic_L_wt_one_plus} gives those for the even branches. In
particular,
\[
\mu^{\an}(f,j)=0
\qquad (0\leq j\leq p-2).
\]
The contragredient branch in
\eqref{eq:contragredient-versus-standard-branch} is obtained from a
standard branch by relabeling $j\mapsto-j$ and applying the Iwasawa
involution $X\mapsto(1+X)^{-1}-1$. This involution preserves coefficient
ideals and the $\mu$- and $\lambda$-invariants.

We now give the coefficient argument directly. Write the scalar
universal branch determined by the appropriate sign basis in
\eqref{eq:isom_H_h} as
\[
 \mathscr L_j(X)=\sum_{n\geq0}a_{n,j}X^n
 \in\gh_{\theta,\gm}\lsem X\rsem,
\]
and put $d=\lambda^{\an}(f,j)$. Since $\mu^{\an}(f,j)=0$, the
specialization at $f$ satisfies
\[
 a_{n,j}(f)\in\pi\cO\quad(n<d),
 \qquad
 a_{d,j}(f)\in\cO^\times.
\]
The specialization at $f$ is local, so the first condition implies
$a_{n,j}\in\gm$ for $n<d$, whereas the second implies
$a_{d,j}\notin\gm$. Because $\gh_{\theta,\gm}$ is local,
$a_{d,j}$ is a unit. Every arithmetic specialization $g$ under
consideration is also local. Consequently,
\[
 a_{n,j}(g)\text{ is a nonunit for }n<d,
 \qquad
 a_{d,j}(g)\in\cO_g^\times.
\]
It follows at once that
\[
 \mu_N^{\an}(g,j)=0,
 \qquad
 \lambda_N^{\an}(g,j)=d=\lambda^{\an}(f,j).
\]
This is the coefficient-content argument underlying
\cite[Section~3.7]{emerton-pollack-weston}; no irreducibility
hypothesis from the broader setting of that reference is used here.
Changing a fixed sign basis multiplies $\mathscr L_j$ by a unit and
does not affect the argument. Substituting the four cases from the two
factorization theorems gives the displayed formulas. The final
primitive-level assertion follows from the comparison in
Section~\ref{sec:MK_p_adic_L}.
\end{proof}
\section{Selmer groups}\label{sec:sel_gp}
We use the cohomological lattice in the ordinary cuspidal family
through $f$. We assume that the cuspidal Hecke algebra is Gorenstein
We describe the ordinary and strict local conditions, identify the
residual ordinary line, and compare the residual and ordinary Selmer
groups.

Let $g$ be either $f$ or a classical ordinary cuspidal specialization
of weight $k\geq2$ as in Theorem~\ref{thm:intro-higher-weight}. Write $V_g$ for its
two-dimensional Galois representation over $K_g$, $\cO_g$ for the
ring of integers of $K_g$, $\pi_g$ for a uniformizer of this ring, and
$\F=\cO_g/\pi_g$ for the residue field.

\subsection{Definition of Selmer groups}\label{sec:def_sel_gp}
For a classical specialization, let $T_g$ be a stable lattice with
the filtration induced by the ordinary line in $V_g$. At $f$ we use
the ordinary line supplied by the family
(Section~\ref{sec:choice_of_lattice}). Twisting by $\omega_p^j$ gives
\begin{equation}\label{eq:ord_fil}
0\longrightarrow T'_{g,j}\longrightarrow T_{g,j}:=T_g\otimes\omega_p^j
\longrightarrow T''_{g,j}\longrightarrow0 .
\end{equation}
Here $T''_{g,j}\simeq\cO_g(\eta_g\omega_p^j)$, where $\eta_g$ is
unramified and sends arithmetic Frobenius to $a_p(g)$. The submodule
has character $\det(V_g)\eta_g^{-1}\omega_p^j$.
Write $E_g=K_g/\cO_g$. Tensoring the sequence with $E_g$ over
$\cO_g$ gives the sequence of discrete $\cO_g[\rG_{\Q_p}]$-modules
\begin{equation}\label{eq:ord_cond}
0\to A'_{g,j} \to A_{g,j}\to A''_{g,j}\to 0.
\end{equation}

Thus $A_{g,j}=V_{g,j}/T_{g,j}$, where
$V_{g,j}=V_g\otimes\omega_p^j$, and $A'_{g,j}$ and $A''_{g,j}$
are defined similarly. For a finite place $v$ of $\Q_\infty$, let
$\rI_{\infty,v}$ be the inertia subgroup of $\rG_{\Q_{\infty,v}}$.
Let $\Sigma_{Np}$ denote the set of primes above $Np$ and let $w$ be the unique
prime above $p$. For $A=A_{g,j}$, put
\[
 \begin{aligned}
 \cH_v(A)&=\mathrm{im}\bigl(H^1(\Q_{\infty,v},A)
                 \to H^1(\rI_{\infty,v},A)\bigr) &&(v\nmid p),\\
 \cH_p(A)&=\mathrm{im}\bigl(H^1(\Q_{\infty,w},A)
                 \to H^1(\rI_{\infty,w},A''_{g,j})\bigr).
 \end{aligned}
\]
The Greenberg Selmer group \cite{greenberg89} is
\begin{equation}\label{eq:def_sel_gp}
 \Sel(\Q_\infty,A)
 =\ker\left(H^1(\rG_{\Q_{\infty},\Sigma_{Np}},A)
 \xrightarrow{\gamma}
 \prod_{v\mid N}\cH_v(A)\times\cH_p(A)\right).
\end{equation}
Replacing inertia in the definition of $\cH_p(A)$ by
$\rG_{\Q_{\infty,w}}$ defines the strict Selmer group, denoted by
$\Sel^{\mathrm{st}}(\Q_\infty,A)$. The local conditions away from
$p$ are unchanged.

Let $X_{g,j}=\Sel(\Q_\infty,A_{g,j})^\vee$ and
$\Lambda_{\cO_g}=\cO_g\lsem\Gamma\rsem$, where
$\Gamma=\Gal(\Q_\infty/\Q)$. If $X_{g,j}$ is torsion, the structure
theorem for finitely generated $\Lambda_{\cO_g}$-modules gives a
pseudo-isomorphism
\[
 X_{g,j}\sim\bigoplus_r\Lambda_{\cO_g}/(\pi_g^{n_r})
         \oplus\bigoplus_i\Lambda_{\cO_g}/(P_i)
\]
with $n_r>0$ and $P_i$ distinguished polynomials. Its characteristic
ideal is
\begin{equation}\label{eq:char_poly}
 \Char_{\Lambda_{\cO_g}}X_{g,j}
   =\left(\pi_g^{\sum_r n_r}\prod_i P_i\right).
\end{equation}
In particular,
\begin{equation}\label{eq:def_alg_inv}
 \mu^{\alg}(A_{g,j})=\sum_r n_r,\qquad
 \lambda^{\alg}(A_{g,j})=\sum_i\deg P_i .
\end{equation}

\subsection{Choice of lattices}\label{sec:choice_of_lattice}
Write $R=\gh_{\theta,\gm}$. The ordinary filtration before the Tate
twist is the exact sequence of $R[\rG_{\Q_p}]$-modules
\[
0\longrightarrow H_{sub,\theta,\gm}\longrightarrow H_{\theta,\gm}
\longrightarrow H_{quo,\theta,\gm}\longrightarrow0.
\]
The submodule is free of rank one over $R$, and the quotient is a
dualizing $R$-module \cite[\S~1.7.13]{FK12}. Since $R$ is Gorenstein,
the quotient is also free of rank one. The underlying sequence of
$R$-modules therefore splits and remains exact after specialization.
Via the specialization $R\to\cO_g$ attached to $g$, we take
\[
 T_g=H_{\theta,\gm}(1)\otimes_R\cO_g,\qquad
 T_{g,j}=T_g\otimes\omega_p^j.
\]
The filtration used throughout is therefore
\begin{equation}\label{eq:ord_fil_weigh_k_form}
0\longrightarrow (H_{sub,\theta,\gm}(1)\otimes_R\cO_g)\otimes\omega_p^j
\longrightarrow T_{g,j}
\longrightarrow (H_{quo,\theta,\gm}(1)\otimes_R\cO_g)\otimes\omega_p^j
\longrightarrow0.
\end{equation}
At $f$ the local representation is scalar, and the family specifies
the ordinary line in $V_f$.

We now record the integral decomposition of the cohomological lattice.
Put
\[
 \fk:=\gh_{\theta,\gm}/\gm,
 \qquad
 \overline H:=H_{\theta,\gm}(1)/\gm H_{\theta,\gm}(1),
\]
and set
\[
 \overline H_-:=H^-_{\theta,\gm}(1)/\gm H^-_{\theta,\gm}(1),
 \qquad
 \overline H_+:=H^+_{\theta,\gm}(1)/\gm H^+_{\theta,\gm}(1),
\]
and let $\overline H_{sub}$ be the image of
$H_{sub,\theta,\gm}(1)$ in $\overline H$.

\begin{prop}\label{residual_repr_is_semisimple}
Suppose that the Hecke algebra $\gh_{\theta,\gm}$ is Gorenstein. Set
\[
 T_f:=H_{\theta,\gm}(1)/\gp_f H_{\theta,\gm}(1),
 \qquad
 V_f:=T_f\otimes_{\cO_f}K_f.
\]
Then
\[
 V_f\simeq K_f(\mathbbm{1})\oplus K_f(\chi^{-1})
 \quad\text{as $K_f[\rG_{\Q}]$-modules},
\]
and there are integral decompositions
\[
 T_f=\cO_fv_-\oplus\cO_fv_+,
 \qquad
 \overline H=\overline H_-\oplus\overline H_+.
\]
Here $\rG_{\Q}$ acts through $\mathbbm{1}$ on the minus summands
$\cO_fv_-$ and $\overline H_-$, and through $\chi^{-1}$ on the plus
summands $\cO_fv_+$ and $\overline H_+$. Choose generators
\[
 e_-\in\overline H_-,\qquad e_+\in\overline H_+,
\]
and let $g$ be any specialization on the chosen cuspidal Hida component.
Put $\F_g=\cO_g/\pi_g$. After scalar extension from $\fk$ to
$\F_g$, one has
\[
 A_{g,j}[\pi_g]\simeq
 \F_g(\omega_p^j)e_-\oplus
 \F_g(\chi^{-1}\omega_p^j)e_+.
\]
Write
\[
 \overline H_{sub}=\fk(ae_-+be_+)
 \qquad ([a:b]\in\mathbf P^1(\fk)).
\]
Then the specialization of the ordinary filtration satisfies
\[
 A'_{g,j}[\pi_g]
 =\F_g(ae_-+be_+)\otimes\omega_p^j.
\]
In particular, the residual representation is a direct sum of global
characters and is scalar on $\rG_{\Q_p}$.
\end{prop}

\begin{proof}
The prime $\gp_f$ contains $I_\theta$ and satisfies
$\gp_f\cap\Lambda_{\cO}^{\mathrm{wt}}=\gp_1$. Consequently, the weight-one specialization
$V_f$ is the scalar extension of the localized Eisenstein quotient
appearing in \cite[Theorem~1.5]{shih-wang}. That theorem applies because
$p\geq5$ and shows that $V_f$ is the direct sum of its trivial and
$\chi^{-1}$ character lines. Let $c\in\rG_{\Q}$ be complex
conjugation. Since $\chi$ is odd and $p$ is odd, the idempotents
\[
 \frac{1+\rho_f(c)}{2},\qquad \frac{1-\rho_f(c)}{2}
\]
preserve $T_f$ and cut out its two character lines. The subscripts
$-$ and $+$ refer to the signs on $H$ \emph{before} twisting:
complex conjugation acts on $\Z_p(1)$ by $-1$, so that the trivial
line of $T_f$ comes from $H^-$ and the $\chi^{-1}$-line from $H^+$.
Reduction modulo $\pi_f$ therefore gives
\[
 \overline H\simeq\fk(\mathbbm{1})\oplus\fk(\chi^{-1}).
\]

Every height-one specialization through $\gm$ has the same residual
representation. More precisely,
\[
 A_{g,j}[\pi_g]\simeq T_{g,j}/\pi_gT_{g,j}
 \simeq(\overline H\otimes_{\fk}\F_g)\otimes\omega_p^j.
\]
The asserted description of $A'_{g,j}[\pi_g]$ follows in the same way
from the specialization of \eqref{eq:ord_fil_weigh_k_form}. Finally,
$\chi$ is unramified
at $p$ and $\chi(p)=1$, so both residual character lines have the same
restriction to $\rG_{\Q_p}$.
\end{proof}

\begin{prop}\label{ordinary_line_normalization}
Assume that the cuspidal Hecke algebra is Gorenstein. The generators $e_-$ and $e_+$ in
Proposition~\ref{residual_repr_is_semisimple} can be chosen, uniquely up
to multiplication by a common element of $\fk^{\times}$, so that
\[
 \overline H_{sub}=\fk(e_-+e_+).
\]
For every specialization $g$ on the chosen cuspidal Hida component,
one then has
\[
 A'_{g,j}[\pi_g]
 =\F_g(e_-+e_+)\otimes\omega_p^j.
\]

Moreover, the generators can be lifted at the weight-one
specialization to a basis $\{\widetilde e_-,\widetilde e_+\}$ for
$T_f$ such that
\[
 T_f=\cO_f\widetilde e_-\oplus\cO_f\widetilde e_+,
 \qquad
 T'_f:=H_{sub,\theta,\gm}(1)/\gp_f H_{sub,\theta,\gm}(1)
      =\cO_f(\widetilde e_-+\widetilde e_+).
\]
Consequently,
\[
 A_{f,j}\simeq
 (K_f/\cO_f)(\omega_p^j)\oplus
 (K_f/\cO_f)(\chi^{-1}\omega_p^j),
 \qquad
 A'_{f,j}=(K_f/\cO_f)(\widetilde e_-+\widetilde e_+)
            \otimes\omega_p^j.
\]
\end{prop}

\begin{proof}
The specialized ordinary filtration is an exact sequence of \(G_{\Q_p}\)-modules
\begin{equation}\label{eq:q}
	0\longrightarrow T'_f\longrightarrow T_f
	\rightarrow Q_f\longrightarrow0.
\end{equation}

Write
\begin{equation}\label{eq:coordinates}
	w_f=a_fv_-+b_fv_+,\qquad a_f,b_f\in\mathcal{O}_f.
\end{equation}
We must show that $a,b\in\fk^{\times}$.
We first show that the reductions of \(a_f\) and \(b_f\) in \(k\) cannot both vanish.
In particular, at least one of \(a_f,b_f\) is a unit. Indeed, since \(Q_f\) is free, reduction of \eqref{eq:q} gives an injection
\(T'_f/\pi T'_f\hookrightarrow T_f/\pi T_f\). Its source is generated
by the nonzero class of \(w_f\), whose image has coordinates
\((\bar a_f,\bar b_f)\). Thus $((a_f,b_f)=\mathcal{O}_f$.

We distinguish the following three cases.
\begin{enumerate}
	\item $\bar a_f\in\fk^{\times}, \bar b_f=0$.
	\item $\bar a_f=0,\bar b_f\in\fk^{\times}$.
	\item $\bar a_f,\bar b_f\in\fk^{\times}$.
\end{enumerate}

It suffices to compute ${}^{0}\mathcal{M}(x)\mod \pi_g$ and
${}^{0}\mathcal{M}(y)\mod\pi_g$, where $x,y$ are generators of
$H_-:=H^-_{\theta,\gm}(1)/IH^-_{\theta,\gm}$ and
$H_+:=H^{+}_{\theta,\gm}(1)/I H^{+}_{\theta,\gm}$, respectively.
Here ${}^{0}\mathcal{M}$ is the map defined in \cite[4.3.3]{FK12}.
By \cite[Theorem~4.4.3]{FK12} and the vanishing of the analytic
$\mu$-invariant, the first two cases cannot occur.

Write $\overline H_{sub}=\fk(ae_-+be_+)$ relative to arbitrary
generators of the two character lines.
Since $a,b\in\fk^{\times}$, rescaling the two generators
separately gives
$\overline H_{sub}=\fk(e_-+e_+)$. This normalization is unique up to
a common scalar. Since the ordinary filtration
\eqref{eq:ord_fil_weigh_k_form} also specializes from the big
filtration, scalar extension to $\F_g$ gives
$A'_{g,j}[\pi_g]=\F_g(e_-+e_+)\otimes\omega_p^j$ for every $g$.

Finally, choose generators $v_-$ and $v_+$ of the character lines
lifting the normalized $e_-$ and $e_+$. Since $T'_f$ is a direct
summand, it has a generator lifting $e_-+e_+$. Write this generator
as $a_fv_-+b_fv_+$. Both coefficients reduce to $1$, so they are
units of $\cO_f$. Replacing $v_-$ by $a_fv_-$
and $v_+$ by $b_fv_+$ gives the claimed basis
$\{\widetilde e_-,\widetilde e_+\}$ and the stated decomposition of
$A_{f,j}$ and $A'_{f,j}$.
\end{proof}

\begin{remark}
We choose the generator $e_{-}+e_{+}$ for convenience. One may
retain the coefficients $a$ and $b$, in which case they appear in the
algebraic calculations in the following section.
\end{remark}

We record the following lemma for future use.
\begin{lemma}\label{cond_show_surj_gamma}
For every specialization $g$,
\[
 \dim_{\F}H^0(\Q_\infty,A_{g,j}[\pi_g])
 =\begin{cases}1& j=0,\\0&j\ne0.\end{cases}
\]
\end{lemma}
\begin{proof}
The residual characters are $\omega_p^j$ and $\chi^{-1}\omega_p^j$
by Proposition~\ref{residual_repr_is_semisimple}. A character of
order prime to $p$ is trivial on $\rG_{\Q_\infty}$ only if it is
trivial on $\rG_\Q$, since $\Gamma$ is pro-$p$. The first character
is trivial precisely at $j=0$; the second is never trivial, since
$\chi$ is nontrivial of conductor prime to $p$.
\end{proof}

\subsection{Finite submodules and residual control}

\begin{prop}\label{sel_no_fin_submod}
Suppose that $X_{g,j}$ is torsion over $\Lambda_{\cO_g}$ and that
either $g$ has weight $k\geq2$ or $g=f$ and $j>0$. Then $X_{g,j}$
has no nonzero finite $\Lambda_{\cO_g}$-submodule.
\end{prop}

\begin{proof}
Set $T=T_{g,j}$, $A=A_{g,j}$, and
$T^\dagger=\Hom_{\cO_g}(T,\cO_g(1))$. Consider the compact
cyclotomic Selmer complex $C$ for $T^\dagger$, with unramified
conditions outside $p$ and the annihilator of $T'_{g,j}$ as the
ordinary line at $p$. We use the perfectness, specialization, and
duality results of \cite[8.9.6.2 and Propositions~9.7.2--9.7.3]{nekovar-selmer}.
We compute the extreme cohomology groups on the residual fiber.

At each $\ell\mid N$, the residual and characteristic-zero
representations have the same conductor exponent, namely
$\operatorname{ord}_\ell N$. Indeed, $\chi$ remains primitive on
reduction since $p\nmid\phi(N)$, and the classical forms have exact
tame conductor $N$. At $f$ the assertion follows from the split
representation. The equality also holds for $T^\dagger$, since
duality preserves the conductor and the cyclotomic twist is
unramified at $\ell$. Invariance of the Swan conductor on reduction
then gives
\[
 \dim_{K_g}(T^\dagger\otimes K_g)^{I_\ell}
   =\dim_{\F}(T^\dagger/\pi_g)^{I_\ell}
\]
as in \cite[Lemma~4.1.2]{emerton-pollack-weston}. Thus the inertia
invariants commute with reduction, for $T$ as well as for
$T^\dagger$. In particular, the unramified local complexes reduce
to the unramified complexes of the residual lattices, with no
Tamagawa error.

Write $\bar T=T/\pi_g$ and $\bar T^\dagger=\Hom_{\F}(\bar T,\F(1))$.
The derived
reduction $C\otimes^{\mathbf L}_{\Lambda_{\cO_g}}\F$ is the
ordinary Selmer complex over $\Q$ for $\bar T^\dagger$.
Its $H^0$ consists of the global invariants that annihilate the
residual ordinary line at $p$. The only possible global invariant
occurs at $j=1$, on the dual of the $e_-$-line after twisting by
$\omega_p$. It does not annihilate $ae_-+be_+$ because $a\ne0$.
Thus $H^0=0$. By finite-coefficient Poitou--Tate duality, $H^3$
is dual to the global invariants of $\bar T$ lying in the residual
ordinary line. The only possible invariant is on the $e_-$-line at
$j=0$. It does not lie in the ordinary line because $b\ne0$.
Thus $H^3=0$ as well.

Since $C$ is perfect in degrees $0$ through $3$, a minimal free
representative over the local ring $\Lambda_{\cO_g}$ now has the
form $[P^1\to P^2]$. These free modules have the same rank.
Indeed, the Euler characteristic of its residual Selmer complex is
zero. The global Euler characteristic is $-1$ by oddness, the
subtraction of the local complex at $p$ for the rank-one ordinary
quotient contributes $+1$, and the terms at the tame primes have
Euler characteristic zero.

By Iwasawa duality, $H^2(C)$ is the Pontryagin dual, up to the
standard involution, of the extended ordinary Selmer group for $A$.
The comparison sequence \cite[(9.6.5.1)]{nekovar-selmer} gives
\[
 A^{\rG_{\Q_\infty}}\longrightarrow
 (A''_{g,j})^{\rG_{\Q_{\infty,w}}}\longrightarrow
 \widetilde H^1_f(\Q_\infty,A)\longrightarrow
 \Sel^{\mathrm{st}}(\Q_\infty,A)\longrightarrow0 .
\]
The local invariant term is cofinitely generated over $\cO_g$,
so its dual is $\Lambda_{\cO_g}$-torsion. The strict Selmer group
is contained in the usual Selmer group, which is cotorsion by
hypothesis. Consequently $H^2(C)$ is torsion.
The square map $P^1\to P^2$ is therefore injective and its cokernel
has projective dimension at most one. A module with such a
presentation over the two-dimensional regular local ring
$\Lambda_{\cO_g}$ has no nonzero finite submodule (by the depth
lemma). The same is true of the dual of the strict Selmer group,
which embeds in this cokernel.

It remains to compare strict and ordinary conditions. For $j>0$,
$(A''_{g,j})^{\rI_{\infty,w}}=0$, so they agree by
inflation--restriction. For $j=0$ and a classical specialization in
our setting, $a_p(g)\ne1$. The ordinary root at a level
prime to $p$ has complex absolute value $p^{(k-1)/2}$ by purity. The
same absolute value holds for a newform of level $Np$ with
nebentype of conductor $Np$ by \cite[Theorem~4.6.17]{miyake}.
Thus Frobenius minus one acts surjectively on $A''_{g,0}$, and
$H^1_{\mathrm{ur}}(\Q_{\infty,w},A''_{g,0})=0$. The two local
conditions again agree.
\end{proof}

\begin{cor}\label{induction_to_residual_sel}
Suppose $X_{g,j}$ is torsion. Then $\mu^{\alg}(A_{g,j})=0$ if
and only if $\Sel(\Q_\infty,A_{g,j})[\pi_g]$ is finite. If these
equivalent conditions hold and $X_{g,j}$ has no nonzero finite
submodule, then $\Sel(\Q_\infty,A_{g,j})$ is $\cO_g$-divisible and
\[
 \lambda^{\alg}(A_{g,j})
   =\dim_{\F}\Sel(\Q_\infty,A_{g,j})[\pi_g].
\]
\end{cor}
\begin{proof}
The dual of the group of $\pi_g$-torsion points is
$X_{g,j}/\pi_gX_{g,j}$. By the structure theorem, this quotient is
finite precisely when $\mu^{\alg}(A_{g,j})=0$. In this case $X_{g,j}$
is finitely generated over $\cO_g$, of rank $\lambda^{\alg}(A_{g,j})$.
Its $\cO_g$-torsion submodule is finite and $\Gamma$-stable. Under
the additional hypothesis, this submodule vanishes, so $X_{g,j}$ is $\cO_g$-free
of the asserted rank.
\end{proof}

\subsection{Residual Selmer groups}\label{sec:res_sel_gp}
For $B=A_{g,j}[\pi_g]$, define $\Sel(\Q_\infty,B)$ by the same
kernel as in \eqref{eq:def_sel_gp}, with $A$ replaced by $B$ and
the local quotient $A''_{g,j}$ replaced by $A''_{g,j}[\pi_g]$.
The lattice determines this local quotient, since divisibility of
$A'_{g,j}$ gives an exact sequence
\[
 0\longrightarrow A'_{g,j}[\pi_g]\longrightarrow B
 \longrightarrow A''_{g,j}[\pi_g]\longrightarrow0.
\]

\begin{lemma}\label{inj_selcond_at_p}
For any $v\mid N$, the map
\[
H^1(\rI_{\infty,v},A_{g,j}[\pi_g])\to H^1(\rI_{\infty,v},A_{g,j})
\]
is injective. At the unique prime $w$ above $p$, the map
\[
H^1(\rI_{\infty,w}, A''_{g,j}[\pi_g])\to H^1(\rI_{\infty,w},A''_{g,j})
\]
is also injective.
\end{lemma}

\begin{proof}
By the multiplication-by-$\pi_g$ sequence, the two kernels are
$A_{g,j}^{\rI_{\infty,v}}/\pi_g$ and
$(A''_{g,j})^{\rI_{\infty,w}}/\pi_g$, respectively. At $v\mid N$,
the conductor equality in the proof of
Proposition~\ref{sel_no_fin_submod} gives
\[
 \dim_{K_g}V_{g,j}^{\rI_{\infty,v}}
 =\dim_\F (A_{g,j}[\pi_g])^{\rI_{\infty,v}}.
\]
Thus the image of $V_{g,j}^{\rI_{\infty,v}}$ in
$A_{g,j}^{\rI_{\infty,v}}$ exhausts its $\pi_g$-torsion and is
divisible, so it equals the full invariant group. At $w$, the
inertia invariants of the quotient are $A''_{g,0}$ if $j=0$ and
zero otherwise. In both cases the invariants are divisible.
\end{proof}

\begin{prop}\label{SES1}
For every specialization $g$, there is a short exact sequence
\begin{equation}\label{eq:residual_sel_and_usual_selmer}
0\to H^0(\Q_{\infty},A_{g,j})/\pi_g \to \Sel(\Q_{\infty},A_{g,j}[\pi_g]) \to \Sel(\Q_{\infty},A_{g,j})[\pi_g] \to 0 .
\end{equation}
If $g$ has weight $k\geq2$, then
\begin{equation}\label{eq:dim_Qinfty_pts}
\dim_{\F} H^0(\Q_{\infty},A_{g,j})/\pi_g
=\begin{cases}
1 & \mbox{ if } j=0\\
0 & \mbox{ if } j\neq 0.
\end{cases}
\end{equation}
If $g=f$, then
\begin{equation}\label{eq:dim_Qinfty_pts_f}
 H^0(\Q_{\infty},A_{f,j})/\pi_f=0
 \qquad(0\leq j\leq p-2).
\end{equation}
\end{prop}

\begin{proof}
Write $G_\Sigma=\rG_{\Q_\infty,\Sigma_{Np}}$ and $A=A_{g,j}$.
Multiplication by $\pi_g$ on $A$ gives
\begin{equation}\label{eq:global-residual-control}
0\longrightarrow H^0(G_\Sigma,A)/\pi_g
\longrightarrow H^1(G_\Sigma,A[\pi_g])
\longrightarrow H^1(G_\Sigma,A)[\pi_g]\longrightarrow0 .
\end{equation}
By Lemma~\ref{inj_selcond_at_p}, the residual local targets inject
into the corresponding targets for $A$ at every prime above $Np$.
Hence a lift of a $\pi_g$-torsion class through
\eqref{eq:global-residual-control} satisfies the residual local
conditions if and only if the class satisfies the ordinary ones.
Restricting \eqref{eq:global-residual-control} to the two Selmer
groups proves \eqref{eq:residual_sel_and_usual_selmer}.

It remains to compute the left-hand term. Put
\[
M_j:=H^0(\Q_\infty,A_{g,j}).
\]
Since $A_{g,j}[\pi_g]$ is the kernel of multiplication by $\pi_g$, we
have
\begin{equation}\label{eq:resdual_p_torsion_pts}
M_j[\pi_g]=H^0(\Q_\infty,A_{g,j}[\pi_g]).
\end{equation}
If $j\neq0$, Lemma~\ref{cond_show_surj_gamma} gives
$M_j[\pi_g]=0$. As $M_j$ is $\pi_g$-primary, this implies $M_j=0$.

Now suppose that $j=0$. If $g=f$, Proposition~\ref{residual_repr_is_semisimple}
gives
\[
 H^0(\Q_\infty,A_{f,0})\simeq K_f/\cO_f,
\]
because the second global character is nontrivial. This module is
divisible, and hence its quotient by $\pi_f$ is zero. This proves
\eqref{eq:dim_Qinfty_pts_f}. We may therefore assume that $g$ has
weight $k\geq2$. The $\cO_g$-module $M_0$ is cofinitely
generated. We claim that its maximal divisible submodule $D_0$ is zero.
Indeed, if $D_0\neq0$, then its $\pi_g$-adic Tate module gives a nonzero
subspace
\[
T_{\pi_g}(D_0)\otimes_{\cO_g}K_g
\subseteq V_{g,0}^{\rG_{\Q_\infty}}.
\]
The group $\rG_{\Q_\infty}$ is normal in $\rG_\Q$, so the invariant
subspace on the right is $\rG_\Q$-stable. This contradicts the absolute
irreducibility of the characteristic-zero representation attached to
the cuspidal eigenform $g$. Indeed, any nonzero invariant subspace
would have to equal $V_{g,0}$, so the representation would factor
through the abelian group $\Gal(\Q_\infty/\Q)$. Thus $D_0=0$, and
consequently $M_0$ is finite.

For a finite module over the discrete valuation ring $\cO_g$, the
elementary divisor decomposition gives
\[
\dim_\F(M_0/\pi_gM_0)=\dim_\F M_0[\pi_g].
\]
Combining this equality with \eqref{eq:resdual_p_torsion_pts} and
Lemma~\ref{cond_show_surj_gamma} gives
\[
\dim_\F H^0(\Q_\infty,A_{g,0})/\pi_g=1,
\]
and proves \eqref{eq:dim_Qinfty_pts}.
\end{proof}

\section{Algebraic invariants and main conjectures}\label{sec:alg_inv}
We retain the cuspidal Gorenstein hypothesis throughout this section. Thus Proposition~\ref{ordinary_line_normalization}
gives character bases in which the ordinary line is generated by
$e_-+e_+$, both at $f$ and residually throughout the family. Let $g$
be a classical ordinary specialization of weight $k\geq2$ as in
Section~\ref{sec:p_adic_L_fcn_higher_wt}, primitive at tame level $N$.
The tame part of the nebentype has conductor $N$, so such a
specialization cannot arise from a proper divisor of the tame level.
In the higher-weight characteristic-ideal comparisons we further
assume that $g$ is the ordinary $p$-stabilization of a newform of level
$N$, so that the underlying newform has good ordinary reduction at $p$
in the sense of \cite[Proposition~17.1]{kato04}.

We first compute the Selmer groups at $f$. To pass to $g$ we use the
common residual representation and ordinary line, keeping the correction
from global invariants on the branch $j=0$. The comparison with the
analytic factorizations and Kato's divisibility then gives the
main conjecture after inverting $p$.

Unless otherwise indicated, character components of class field
modules are taken after extension of scalars to $\cO$, and $\Hom$
denotes continuous $\cO$-linear homomorphisms (with the induced
$\cO$-action on a residue field). We identify the dual of $K_f/\cO$
with $\cO$ by choosing a generator of the inverse different.

\subsection{The exceptional even twist}\label{sec:trivial_twist}
We begin with $j=0$, where the ordinary and strict local conditions
distinguish two algebraic objects. This is the algebraic counterpart of
the even trivial-zero correction in Theorem~\ref{p_adic_L_wt_one_plus}.

\subsubsection{The weight-one Selmer groups}

\begin{prop}\label{comp_sel_wt1_nontwist}
One has
\[
 \Sel(\Q_{\infty},A_{f,0})^{\vee}
 \simeq X_{Np^{\infty}}^{(\chi^{-1})},
 \qquad
 \Sel^{\mathrm{st}}(\Q_{\infty},A_{f,0})^{\vee}
 \simeq X_{Np^{\infty},p}^{(\chi^{-1})}.
\]
In particular, $\Sel(\Q_{\infty},A_{f,0})^{\vee}$ is a torsion
$\Lambda$-module, $\mu^{\alg}(A_{f,0})=0$, and
\eqref{eq:sel_gp_weight_one_nontwist} holds for $j=0$.
\end{prop}

\begin{proof}
Recall from Section~\ref{sec:choice_of_lattice} that the action of $\rG_{\Q}$ on $A_{f,0}$ is given by the split sequence
\begin{equation}\label{eq:wt1_global_seq}
0\to K_f/\cO_f\to A_{f,0} \to K_f/\cO_f(\chi^{-1}) \to 0.
\end{equation}
In the normalized character basis the ordinary quotient sends
$(x,y)$ to $x-y$ after identifying the two local character lines
at $p$. Thus the local condition on a pair of global classes is
\begin{equation}\label{eq:wt1-local-matching}
 (x-y)|_{\rI_{\infty,w}}=0.
\end{equation}
The global conditions away from $p$ force the first coordinate $x$ to
factor through the trivial-character quotient of
$\mathfrak X_{p^\infty}$. To see that this quotient is zero, note that the maximal pro-$p$ extension
of $\Q$ unramified outside $p$ is $\Q_\infty$. Indeed, by the
Kronecker--Weber theorem its Galois group's maximal abelian quotient
is $\Gamma\simeq\Z_p$, so the pro-$p$ group is generated by one
element and is itself abelian. The maximal pro-$p$ extension of
$\Q_\infty$ unramified outside $p$ is Galois over $\Q$, with pro-$p$
Galois group, so it too is trivial over $\Q_\infty$. In particular,
$x=0$. Equation~\eqref{eq:wt1-local-matching}
then shows that $y$ is unramified at $p$.  The remaining local conditions
say exactly that
\[
 \Sel(\Q_\infty,A_{f,0})
 \simeq
 \Hom\!\left(X_{Np^\infty}^{(\chi^{-1})},K_f/\cO_f\right).
\]
Taking Pontryagin duals proves the first assertion.  For the
strict Selmer group, the same argument replaces the condition that $y$
be unramified at $p$ by the condition that its restriction to the local
decomposition group vanish.  Hence
\[
 \Sel^{\mathrm{st}}(\Q_\infty,A_{f,0})
 \simeq
 \Hom\!\left(X_{Np^\infty,p}^{(\chi^{-1})},K_f/\cO_f\right),
\]
as required.  Finally, the abelian Iwasawa main conjecture and the
Ferrero--Washington theorem show that
$X_{Np^\infty}^{(\chi^{-1})}$ is torsion with vanishing
$\mu$-invariant.  They also give the asserted characteristic-ideal
formula.
\end{proof}

The distinction between the two modules in
Proposition~\ref{comp_sel_wt1_nontwist} accounts exactly for the factor
$X$ on the even trivial branch.

\begin{lemma}\label{lem:split-at-p-factor}
One has
\begin{equation}\label{eq:split-at-p-factor}
\Char_{\Lambda_{\cO}}X_{Np^{\infty}}^{(\chi^{-1})}
=X\,\Char_{\Lambda_{\cO}}X_{Np^{\infty},p}^{(\chi^{-1})}.
\end{equation}
\end{lemma}

\begin{proof}
Let $\mathcal D_p$ be the closed $\Lambda_{\cO}$-submodule of
$X_{Np^{\infty}}$ generated by the decomposition groups of the primes
above $p$. By the definition of the completely split quotient, there is
an exact sequence
\[
0\longrightarrow \mathcal D_p^{(\chi^{-1})}
\longrightarrow X_{Np^{\infty}}^{(\chi^{-1})}
\longrightarrow X_{Np^{\infty},p}^{(\chi^{-1})}
\longrightarrow0.
\]
Since $\chi^{-1}$ is odd and $\chi^{-1}(p)=1$,
\cite[Lemma~3.1]{shih-wang}, applied with $\chi$ there replaced by
$\chi^{-1}$, gives
\[
\Char_{\Lambda_{\cO}}\mathcal D_p^{(\chi^{-1})}=(X).
\]
Equivalently, the decomposition-group quotient contributes exactly the
simple exceptional factor $X$. Multiplicativity of characteristic ideals
in the exact sequence now gives \eqref{eq:split-at-p-factor}.
\end{proof}

\begin{cor}\label{cor:strict-main-f0}
The strict Selmer group
on the even trivial branch satisfies
\begin{equation}\label{eq:strict-main-f0}
\Char_{\Lambda_{\cO}}
\Sel^{\mathrm{st}}(\Q_{\infty},A_{f,0})^{\vee}
=\cL(f,\mathbbm{1})\Lambda_{\cO}.
\end{equation}
\end{cor}

\begin{proof}
By Proposition~\ref{comp_sel_wt1_nontwist}, the left-hand side of
\eqref{eq:strict-main-f0} is
$\Char_{\Lambda_{\cO}}X_{Np^{\infty},p}^{(\chi^{-1})}$. On the other
hand, Lemma~\ref{lem:split-at-p-factor}, the abelian Iwasawa main
conjecture \eqref{eq:classical_main_conj}, and the analytic formula
\eqref{eq:i_branch} give
\[
\Char_{\Lambda_{\cO}}X_{Np^{\infty},p}^{(\chi^{-1})}
=\frac{G_{\chi\omega_p}(u(1+X)^{-1}-1)}{X}\Lambda_{\cO}
=\cL(f,\mathbbm{1})\Lambda_{\cO},
\]
as desired.
\end{proof}

\subsubsection{The higher-weight specialization}

We next compute the residual Selmer group for the higher-weight form
$g$ and extract its algebraic invariants.

\begin{prop}\label{res_sel_gp_trivial_twist}
One has
\[
\dim_{\F} \Sel(\Q_{\infty}, A_{g,0}[\pi_g])
= \dim_{\F} \Hom(X_{Np^{\infty}}^{(\chi^{-1})},\F).
\]
Moreover, $\mu^{\alg}(A_{g,0})=0$, and
\eqref{eq:inv_higher_wt} holds for $j=0$.
\end{prop}

\begin{proof}
In the normalized residual character basis the local condition on
$(x,y)$ is $(x-y)|_{\rI_{\infty,w}}=0$. The trivial-character component
$x$ vanishes by the argument of Proposition~\ref{comp_sel_wt1_nontwist}.
Thus the local condition requires $y$ to be unramified at $p$.
The conditions away from $p$ impose unramifiedness at the remaining
finite primes. By class field theory, the resulting residual group is
$\Hom(X_{Np^\infty}^{(\chi^{-1})},\F)$. The odd component
$X_{Np^\infty}^{(\chi^{-1})}$ has no finite $\Lambda$-submodule
\cite[Proposition~13.28]{washington-GTM83} and has $\mu=0$.
It is therefore $\cO$-free of rank equal to its $\lambda$-invariant.
Proposition~\ref{SES1} and \eqref{eq:dim_Qinfty_pts} subtract one
from this dimension in passing to $\Sel(\Q_\infty,A_{g,0})[\pi_g]$.
Its finiteness implies that the compact dual is finite over $\cO_g$
by the compact Nakayama lemma. In particular, it is $\Lambda_{\cO_g}$-torsion
with $\mu^{\alg}(A_{g,0})=0$.
The absence of finite submodules in Proposition~\ref{sel_no_fin_submod}
then identifies its dimension with $\lambda^{\alg}(A_{g,0})$.
\end{proof}

\begin{remark}\label{rem:even-trivial-zero-algebraic}
The algebraic effect of the even trivial zero can now be seen directly.
Proposition~\ref{comp_sel_wt1_nontwist}, the abelian Iwasawa main
conjecture, and \eqref{eq:i_branch} give
\[
\Char_{\Lambda_{\cO}}
\Sel(\Q_\infty,A_{f,0})^\vee
=
X\,\cL(f,\mathbbm{1})\Lambda_{\cO}.
\]
Thus the ordinary weight-one Selmer group retains precisely the linear
factor removed analytically by the trivial-zero cancellation.
Corollary~\ref{cor:strict-main-f0} shows that the strict local condition
quotients out precisely this decomposition-group contribution and
restores the main-conjecture-type equality. For the higher-weight
specialization $g$, the contribution of global invariants in
Proposition~\ref{res_sel_gp_trivial_twist} lowers the algebraic
$\lambda$-invariant by one, matching the analytic quotient by $X$.
This is the even counterpart of the factor $1+X-u$ on the odd
exceptional branch.
\end{remark}
\subsection{Positive even twists}\label{sec:even_twist}
We now assume that $2\leq j\leq p-2$ is even. Unlike the case $j=0$,
there is no exceptional decomposition-group factor. We express the
ordinary local condition as a matching condition
between the two character components.

\subsubsection{The residual matching calculation}

Put $\alpha_j=\omega_p^j$ and $\beta_j=\chi^{-1}\omega_p^j$.
In the normalized basis of Proposition~\ref{ordinary_line_normalization},
a residual Selmer class is a pair $(x,y)$ with respective characters
$\alpha_j,\beta_j$, subject to the unramified conditions away from $p$
and to
\begin{equation}\label{eq:decom_res_sel_gp}
 x|_{\rI_{\infty,w}}=y|_{\rI_{\infty,w}}.
\end{equation}
The same description applies at weight one, with coefficients
$K_f/\cO_f$ in place of $\F$:
\begin{equation}\label{eq:decom-sel-f-general}
 x|_{\rI_{\infty,w}}=y|_{\rI_{\infty,w}}.
\end{equation}
For $\Omega=\F$ or $K_f/\cO_f$, one has
$$
H^1(\rG_{\Q_{\infty},\Sigma_p},\Omega(\omega_p^j)) \simeq 
\Hom(\mathfrak{X}_{p^{\infty}}^{(\omega_p^j)},\Omega)
$$
and
$$
H^1_{\mathrm{unr}\,(N)}(\rG_{\Q_{\infty},\Sigma_{Np}},\Omega(\chi^{-1}\omega_p^j))\simeq
\Hom(\mathfrak{X}_{Np^{\infty}}^{(\chi^{-1}\omega_p^j)},\Omega).
$$

Here the subscript $\mathrm{unr}\,(N)$ imposes the unramified conditions
at all primes over $N$. These identifications follow by restricting
to the fields cut out by the characters and applying class field theory.

\begin{prop}\label{computation_sel_with_twist}
Let $\Omega$ be $\F$ or $K_f/\cO_f$, and let
$2\leq j\leq p-2$ be even. For every
\[
x\in\Hom(\mathfrak X_{p^\infty}^{(\omega_p^j)},\Omega),
\]
there is a unique class
\[
\bar y\in
\frac{\Hom(\mathfrak X_{Np^\infty}^{(\chi^{-1}\omega_p^j)},\Omega)}
{\Hom(X_{Np^\infty}^{(\chi^{-1}\omega_p^j)},\Omega)}
\]
such that
\[
x|_{\rI_{\infty,w}}=y|_{\rI_{\infty,w}}
\]
for one, and hence for every, lift $y$ of $\bar y$.
\end{prop}

\begin{proof}
Put
\[
\alpha_j:=\omega_p^j,
\qquad
\beta_j:=\chi^{-1}\omega_p^j,
\]
and let $K_{\mathrm{ab}}$ be the maximal abelian pro-$p$ extension of
$\Q_p(\mu_{p^\infty})$. We use the abbreviations
\[
\mathcal P_j:=
\Gal(K_{\mathrm{ab}}/\Q_p(\mu_{p^\infty}))^{(\alpha_j)},
\qquad
D_j:=\Gal(M_\infty/L_\infty)^{(\beta_j)}.
\]
The restrictions to $\rI_{\infty,w}$ in the statement are the maps obtained by
precomposing with the corresponding local reciprocity maps. We first
show that the natural map
\begin{equation}\label{eq:even-local-global-map}
q_j:\mathcal P_j\longrightarrow D_j
\end{equation}
is an isomorphism.

Let $\Q_\chi$ be the field cut out by $\chi$ and put
$F_n=\Q_\chi(\mu_{p^{n+1}})$. Since $\chi(p)=1$, the prime $p$ splits
completely in $\Q_\chi$. For a prime $\mathfrak p$ of $F_0$ above $p$,
let $\mathfrak p_n$ be the unique prime of $F_n$ above $\mathfrak p$, and
set
\[
\mathcal U:=\varprojlim_n
\prod_{\mathfrak p\mid p}U_{F_n,\mathfrak p_n},
\]
where $U_{F_n,\mathfrak p_n}$ is the group of principal local units.
Global class field theory gives
\begin{equation}\label{eq:even-ramification-units}
D_j\simeq(\mathcal U/\overline E)^{(\beta_j)},
\end{equation}
where $\overline E$ is the inverse limit of the semilocal closures of
the global units at the finite layers;
see \cite[Corollary~13.6]{washington-GTM83}. On the other hand, local
class field theory and the induction over the primes above $p$ give
\begin{equation}\label{eq:even-local-units}
\mathcal P_j\simeq\mathcal U^{(\beta_j)};
\end{equation}
see \cite[Proposition~4.12]{iwasawa86}. In this identification the
valuation part of local reciprocity disappears: it has trivial
$\Delta$-action and $\alpha_j\ne\mathbbm{1}$. Under
\eqref{eq:even-ramification-units} and \eqref{eq:even-local-units}, the
map $q_j$ is induced by $\mathcal U\to\mathcal U/\overline E$.

Because $j$ is even and $\chi$ is odd, $\beta_j$ is odd. Moreover, its
tame component $\chi^{-1}$ is nontrivial. By Dirichlet's unit theorem
applied to the CM fields $F_n$, the odd part of the $p$-adic completion
of the global units consists only of the component arising from roots of unity. Its
restriction to $\Gal(\Q_\chi/\Q)$ is trivial, whereas $\beta_j$ has the
nontrivial tame factor $\chi^{-1}$. Consequently,
\[
\overline E^{(\beta_j)}=0.
\]
It follows that $q_j$ in \eqref{eq:even-local-global-map} is an
isomorphism.

The standing hypothesis $p\nmid\phi(N)$ ensures that taking
$\beta_j$-components is exact. We may therefore take the
$\beta_j$-component of the class-field-theoretic exact sequence
\[
0\longrightarrow\Gal(M_\infty/L_\infty)
\longrightarrow\mathfrak X_{Np^\infty}
\longrightarrow X_{Np^\infty}
\longrightarrow0.
\]
Writing
\[
B_j:=\mathfrak X_{Np^\infty}^{(\beta_j)},
\qquad C_j:=X_{Np^\infty}^{(\beta_j)},
\]
we obtain
\begin{equation}\label{eq:even-ramification-exact}
0\longrightarrow D_j\longrightarrow B_j\longrightarrow C_j
\longrightarrow0.
\end{equation}
Ferrero--Washington gives $\mu(C_j)=0$, so $C_j$ is finitely generated
over $\cO$. Its $\cO$-torsion is a finite $\Lambda$-submodule,
hence is zero by \cite[Proposition~13.28]{washington-GTM83}, since
$\beta_j$ is odd. Lifting an $\cO$-basis of $C_j$ therefore gives a
continuous $\cO$-linear section of \eqref{eq:even-ramification-exact}
(no $\Lambda$-linear splitting is asserted). Applying
$\Hom(-,\Omega)$ for either choice of $\Omega$ gives
the short exact sequence
\begin{equation}\label{eq:exist_bar_y}
0\longrightarrow\Hom(C_j,\Omega)
\longrightarrow\Hom(B_j,\Omega)
\longrightarrow\Hom(D_j,\Omega)
\longrightarrow0.
\end{equation}

Now restrict $x$ to $\mathcal P_j$, obtaining
$x_p\in\Hom(\mathcal P_j,\Omega)$. Since $q_j$ is an isomorphism,
$x_p\circ q_j^{-1}$ is a homomorphism $D_j\to\Omega$. By
\eqref{eq:exist_bar_y}, it lifts to some $y\in\Hom(B_j,\Omega)$, and the
definition of $q_j$ gives $y|_{\rI_{\infty,w}}=x|_{\rI_{\infty,w}}$. Any two such lifts
differ by an element of $\Hom(C_j,\Omega)$, and every such element is
unramified at $p$. Hence the class $\bar y$ is unique and every lift of
it has the required restriction.
\end{proof}

\subsubsection{Algebraic invariants and the even main conjectures}

\begin{cor}[Positive even algebraic invariants]
\label{cor:positive-even-algebraic}
Let
$2\leq j\leq p-2$ be even. Then
$\Sel(\Q_{\infty},A_{f,j})^{\vee}$ is a torsion $\Lambda_{\cO}$-module,
$\mu^{\alg}(A_{f,j})=0$, and
\begin{equation}\label{eq:positive-even-char-f}
\Char_{\Lambda_{\cO}}\Sel(\Q_{\infty},A_{f,j})^{\vee}
=\Char_{\Lambda_{\cO}}\mathfrak X_{p^{\infty}}^{(\omega_p^j)}
\cdot
\Char_{\Lambda_{\cO}}X_{Np^{\infty}}^{(\chi^{-1}\omega_p^j)}.
\end{equation}
Moreover,
\begin{equation}\label{eq:positive-even-alg-g}
\mu^{\alg}(A_{g,j})=0,
\qquad
\lambda^{\alg}(A_{g,j})
=\lambda\!\left(\Char_{\Lambda_{\cO}}
\mathfrak X_{p^{\infty}}^{(\omega_p^j)}\right)
+\lambda\!\left(\Char_{\Lambda_{\cO}}
X_{Np^{\infty}}^{(\chi^{-1}\omega_p^j)}\right).
\end{equation}
\end{cor}

\begin{proof}
Applying Proposition~\ref{computation_sel_with_twist} to $x$ in
\eqref{eq:decom-sel-f-general}, with
$\Omega=K_f/\cO_f$, gives an exact sequence
\[
0\longrightarrow
\Hom\!\left(X_{Np^\infty}^{(\chi^{-1}\omega_p^j)},K_f/\cO_f\right)
\longrightarrow \Sel(\Q_\infty,A_{f,j})
\longrightarrow
\Hom\!\left(\mathfrak X_{p^\infty}^{(\omega_p^j)},K_f/\cO_f\right)
\longrightarrow0.
\]
Pontryagin duality gives \eqref{eq:positive-even-char-f}. The theorem of
Ferrero--Washington gives the vanishing of the $\mu$-invariant.

For $g$, the same proposition applied to
\eqref{eq:decom_res_sel_gp}, now with
$\Omega=\F$, yields
\[
\begin{aligned}
\dim_{\F}\mathrm{Sel}(\Q_{\infty},A_{g,j}[\pi_g])
&=\dim_{\F}\Hom(\mathfrak{X}_{p^{\infty}}^{(\omega_p^j)},\F)\\
&\quad+\dim_{\F}\Hom(X_{Np^{\infty}}^{(\chi^{-1}\omega_p^j)},\F)
<\infty.
\end{aligned}
\]
Since $j>0$, Proposition~\ref{SES1} and
\eqref{eq:dim_Qinfty_pts} give
\[
\mathrm{Sel}(\Q_{\infty},A_{g,j}[\pi_g])
\simeq \Sel(\Q_\infty,A_{g,j})[\pi_g].
\]
The same residual calculation applies at $f$ over its own residue
field and gives the same dimension because the two residual Selmer groups are
scalar extensions of the Selmer group of the common residual
representation and ordinary line. Finiteness of the residual group
implies that the dual at $g$ is finitely generated over $\cO_g$, so it
is torsion over $\Lambda_{\cO_g}$ and has $\mu=0$. The dual at $f$ is
already known to be torsion with $\mu=0$. By
Proposition~\ref{sel_no_fin_submod}, neither dual has a finite
submodule; both are therefore free over their respective coefficient
rings. Their ranks are the common residual dimension. Taking
$\lambda$-invariants in \eqref{eq:positive-even-char-f} now gives
\eqref{eq:positive-even-alg-g}.
\end{proof}

\begin{thm}[Main conjectures on the even branches]
\label{thm:even-main-conjectures}
The following statements
hold.
\begin{enumerate}
\item For every positive even $j$ with $2\leq j\leq p-2$,
\begin{equation}\label{eq:positive-even-main-f}
\cL(f,\omega_p^j)\Lambda_{\cO}
=\Char_{\Lambda_{\cO}}
\Sel(\Q_{\infty},A_{f,j})^{\vee}.
\end{equation}

\item On the exceptional branch $j=0$, one has
\begin{equation}\label{eq:even-exceptional-main-summary}
\Char_{\Lambda_{\cO}}\Sel(\Q_{\infty},A_{f,0})^{\vee}
=X\,\cL(f,\mathbbm{1})\Lambda_{\cO},
\qquad
\Char_{\Lambda_{\cO}}\Sel^{\mathrm{st}}(\Q_{\infty},A_{f,0})^{\vee}
=\cL(f,\mathbbm{1})\Lambda_{\cO}.
\end{equation}

\item If $g$ is the ordinary $p$-stabilization of a newform of
level $N$, then for every even $0\leq j\leq p-2$ it satisfies
Conjecture~\ref{main_conj}, so
\begin{equation}\label{eq:even-main-g}
\cL(g,\omega_p^j)\Lambda_{\cO_g}[1/p]
=\Char_{\Lambda_{\cO_g}[1/p]}
\Sel(\Q_{\infty},A_{g,j})^{\vee}.
\end{equation}
\end{enumerate}
\end{thm}

\begin{proof}
For a positive even $j$, Corollary~\ref{cor:positive-even-algebraic}
and the two abelian main-conjecture formulas
\eqref{eq:classical_main_conj} and
\eqref{eq:classical-main-conj-p-ramified} give
\[
\Char_{\Lambda_{\cO}}\Sel(\Q_{\infty},A_{f,j})^{\vee}
=\left(
G_{\omega_p^j}(X)
G_{\chi\omega_p^{1-j}}(u(1+X)^{-1}-1)
\right).
\]
Theorem~\ref{p_adic_L_wt_one_plus} identifies the generator on the
right with $\cL(f,\omega_p^j)$ up to a unit, proving
\eqref{eq:positive-even-main-f}.

For $j=0$, Proposition~\ref{comp_sel_wt1_nontwist},
Lemma~\ref{lem:split-at-p-factor}, and
the abelian main conjecture \eqref{eq:classical_main_conj}, together
with Theorem~\ref{p_adic_L_wt_one_plus}, give the ordinary equality in
\eqref{eq:even-exceptional-main-summary}; the strict equality is
Corollary~\ref{cor:strict-main-f0}.

It remains to treat $g$. Proposition~\ref{res_sel_gp_trivial_twist}
for $j=0$ and Corollary~\ref{cor:positive-even-algebraic} for $j>0$
give $\mu^{\alg}(A_{g,j})=0$ and the algebraic formulas in
\eqref{eq:inv_higher_wt}. Corollary~\ref{ana_inv_(g,j)} gives
$\mu^{\an}(g,j)=0$ and the corresponding analytic formulas. The
abelian main conjecture therefore yields
\[
\lambda^{\an}(g,j)=\lambda^{\alg}(A_{g,j})
\qquad (j\ \text{even}).
\]
Finally, Kato's divisibility \eqref{eq:divisible_kato} shows that the
quotient of the analytic ideal by the algebraic characteristic ideal
has $\lambda$-invariant zero. Any remaining power of $p$ is invertible
in $\Lambda_{\cO_g}[1/p]$, so the quotient is a unit. This proves
\eqref{eq:even-main-g}.
\end{proof}

\subsection{Odd twists}\label{sec:odd_twist}

For odd $j$, put $\alpha_j=\omega_p^j$ and
$\beta_j=\chi^{-1}\omega_p^j$. Now $\alpha_j$ is odd and
$\beta_j$ is even, so the global-unit contribution lies on the opposite
side of the matching condition.

\subsubsection{The weight-one Selmer group}

Let $P_j$ be the $\alpha_j$-part of the inverse limit of the principal
units of $\Q_p(\mu_{p^{n+1}})$ under the norm maps, with coefficients
extended to $\cO$. Fix a prime above $p$ of $\Q_\chi$. Since $p$ splits
completely in $\Q_\chi$, induction over the primes above $p$ identifies
$P_j$ with the $\beta_j$-part of the semilocal units over
$\Q_\chi(\mu_{p^\infty})$. Using these two descriptions, global
reciprocity gives maps
\[
r_j:P_j\longrightarrow\mathfrak X_{p^\infty}^{(\alpha_j)},
\qquad
s_j:P_j\longrightarrow\mathfrak X_{Np^\infty}^{(\beta_j)}.
\]
Here the latter eigenspace can equally be computed over
$\Q_\chi(\mu_{p^\infty})$: passage to the larger conductor-$N$
field has prime-to-$p$ degree and does not change this component.

\begin{lemma}\label{lem:odd-class-field}
There is an exact sequence
\begin{equation}\label{eq:odd-class-field-sequence}
0
\longrightarrow
Z_j
\longrightarrow
P_j
\xrightarrow{\,r_j\,}
\mathfrak X_{p^\infty}^{(\alpha_j)}
\longrightarrow
X_{p^\infty}^{(\alpha_j)}
\longrightarrow
0.
\end{equation}
If $j\neq1$, then
\[
Z_j=0.
\]
If $j=1$, then
\begin{equation}\label{eq:odd-class-field-kernel}
Z_1\simeq\cO(1)
\simeq
\Lambda_{\cO}/(\ell_1(X)),
\qquad
\ell_1(X)=1+X-u.
\end{equation}
Moreover, the restriction
\[
s_1|_{Z_1}:Z_1\longrightarrow
\mathfrak X_{Np^\infty}^{(\chi^{-1}\omega_p)}
\]
is injective.
\end{lemma}

\begin{proof}
The exact sequence \eqref{eq:odd-class-field-sequence} is the usual global-to-local
class-field-theoretic sequence.

Since $\alpha_j$ is odd, the ordinary global-unit component in the
kernel of $r_j$ is zero. The only possible odd contribution is the
norm-compatible system of $p$-power roots of unity, which transforms
by $\omega_p$. Hence
\[
Z_j=0
\quad\text{for }j\neq1,
\]
whereas
\[
Z_1\simeq\cO(1).
\]

If $\gamma$ is the chosen topological generator of $\Gamma$ with
\[
\kappa_p(\gamma)=u,
\]
then $\gamma$ acts on $\cO(1)$ by multiplication by $u$.
Under the identification $\gamma\leftrightarrow1+X$, this yields
\[
Z_1\simeq\Lambda_{\cO}/(1+X-u).
\]

For the last assertion, suppose $z\in Z_1$ lies in $\ker(s_1)$.
At each layer $F_n=\Q_\chi(\mu_{p^{n+1}})$ the semilocal image of $z$
belongs to the $\beta_1$-part of the closure of the global units, by
the class field sequence. This image also has finite order, since its
coordinates are roots of unity. Leopoldt's theorem for the abelian
field $F_n$ \cite{brumer1967} identifies the $p$-adic completion of
the global units with its semilocal image; in particular, torsion in the
image comes from global $p$-power roots of unity. These have trivial
$\Gal(\Q_\chi/\Q)$-action, while the tame part of $\beta_1$ is
nontrivial. Thus the image of $z$ vanishes at every level. The
identification of $P_1$ with the semilocal $\beta_1$-component gives
$z=0$.
\end{proof}

\begin{prop}\label{prop:odd-selmer-f-corrected}
For odd $j\ne1$,
$\operatorname{Sel}(\mathbf Q_\infty,A_{f,j})^\vee$ is a torsion
$\Lambda$-module and there is a short exact sequence
\begin{equation}\label{eq:odd-selmer-sequence}
0
\longrightarrow
\mathfrak X_{Np^\infty}^{(\chi^{-1}\omega_p^j)}
\longrightarrow
\operatorname{Sel}(\mathbf Q_\infty,A_{f,j})^\vee
\longrightarrow
X_{p^\infty}^{(\omega_p^j)}
\longrightarrow
0.
\end{equation}

On the exceptional branch,
$\operatorname{Sel}(\mathbf Q_\infty,A_{f,1})^\vee$ is torsion and
there is a short exact sequence
\begin{equation}\label{eq:odd-selmer-exceptional-sequence}
0
\longrightarrow
\frac{
\mathfrak X_{Np^\infty}^{(\chi^{-1}\omega_p)}
}{
s_1(Z_1)
}
\longrightarrow
\operatorname{Sel}(\mathbf Q_\infty,A_{f,1})^\vee
\longrightarrow
X_{p^\infty}^{(\omega_p)}
\longrightarrow
0.
\end{equation}
Moreover,
\begin{equation}\label{eq:omega-unramified-vanishing}
X_{p^\infty}^{(\omega_p)}=0,
\end{equation}
and therefore
\begin{equation}\label{eq:odd-selmer-exceptional-iso}
\operatorname{Sel}(\mathbf Q_\infty,A_{f,1})^\vee
\simeq
\frac{
\mathfrak X_{Np^\infty}^{(\chi^{-1}\omega_p)}
}{
s_1(Z_1)
}.
\end{equation}
\end{prop}

\begin{proof}
By the normalized character decomposition, the weight-one Selmer group
consists of pairs $(x,y)$ of continuous homomorphisms from
$\mathfrak X_{p^\infty}^{(\alpha_j)}$ and
$\mathfrak X_{Np^\infty}^{(\beta_j)}$, respectively, to $K_f/\cO$,
subject to the single local condition
\begin{equation}\label{eq:odd-selmer-local-condition}
 x\circ r_j=y\circ s_j \qquad\text{on }P_j.
\end{equation}

Since $K_f/\mathcal O_f$ is divisible, Pontryagin duality identifies
the dual Selmer group with the compact cokernel
\begin{equation}\label{eq:odd-selmer-cokernel}
\operatorname{Sel}(\mathbf Q_\infty,A_{f,j})^\vee
\simeq
 \operatorname{coker}
 \left(
 P_j
 \xrightarrow{(r_j,-s_j)}
\mathfrak X_{p^\infty}^{(\alpha_j)}
\oplus
\mathfrak X_{Np^\infty}^{(\beta_j)}
\right).
\end{equation}

For arbitrary homomorphisms
\[
r:P\to A,
\qquad
s:P\to B,
\]
there is a natural exact sequence
\begin{equation}\label{eq:abstract-cokernel-sequence}
0
\longrightarrow
B/s(\ker r)
\longrightarrow
\operatorname{coker}(r,-s)
\longrightarrow
\operatorname{coker}(r)
\longrightarrow
0.
\end{equation}
Applying this with $s=s_j$ to
\eqref{eq:odd-selmer-cokernel} and using
Lemma~\ref{lem:odd-class-field} gives
\begin{equation}\label{eq:odd-selmer-general-sequence}
0
\longrightarrow
 \frac{
 \mathfrak X_{Np^\infty}^{(\beta_j)}
 }{
 s_j(Z_j)
}
\longrightarrow
\operatorname{Sel}(\mathbf Q_\infty,A_{f,j})^\vee
\longrightarrow
X_{p^\infty}^{(\alpha_j)}
\longrightarrow
0.
\end{equation}

If $j\neq1$, then $Z_j=0$, giving
\eqref{eq:odd-selmer-sequence}; its two outer terms are torsion, so the
middle term is torsion as well.
For $j=1$, the same sequence gives
\eqref{eq:odd-selmer-exceptional-sequence}. At the bottom layer
the $\omega_p$-component of the $p$-class group of $\Q(\mu_p)$
vanishes by the classical reflection theorem (its reflected component
is the trivial-character part). There is just one prime above $p$,
totally ramified in the cyclotomic tower; hence the $\Gamma$-coinvariants
of the unramified Iwasawa module identify with the bottom-layer
$p$-class group \cite[Chapter~13]{washington-GTM83}.
Nakayama's lemma gives
\[
X_{p^\infty}^{(\omega_p)}=0,
\]
which proves \eqref{eq:odd-selmer-exceptional-iso}.  The module
$\mathfrak X_{Np^\infty}^{(\chi^{-1}\omega_p)}$ is torsion by the
abelian Iwasawa main conjecture, so its displayed quotient, and hence
the exceptional Selmer dual, is torsion as well.

\end{proof}

\subsubsection{Invariants and the weight-one main conjecture}

We use the following two forms of the abelian Iwasawa main conjecture.
For a nontrivial even character $\eta$, the $p$-ramified form gives
\begin{equation}\label{eq:p-ramified-main-conjecture}
\operatorname{Char}_{\Lambda}
\mathfrak X^{(\eta)}
=
\bigl(G_\eta\bigr),
\end{equation}
whereas for an odd character $\rho\ne\omega_p$ the unramified form gives
\begin{equation}\label{eq:unramified-main-conjecture}
\operatorname{Char}_{\Lambda}
X^{(\rho)}
=
\left(
G_{\rho^{-1}\omega_p}
\bigl(u(1+X)^{-1}-1\bigr)
\right).
\end{equation}

In particular,
\begin{equation}\label{eq:odd-char-p-ramified}
\operatorname{Char}_{\Lambda}
\mathfrak X_{Np^\infty}^{(\chi^{-1}\omega_p^j)}
=
\bigl(
G_{\chi^{-1}\omega_p^j}(X)
\bigr),
\end{equation}
and, for $j\neq1$,
\begin{equation}\label{eq:odd-char-unramified}
\operatorname{Char}_{\Lambda}
X_{p^\infty}^{(\omega_p^j)}
=
\left(
G_{\omega_p^{1-j}}
\bigl(u(1+X)^{-1}-1\bigr)
\right).
\end{equation}

\begin{cor}\label{cor:odd-algebraic-f}
For every odd $j\neq1$, one
has $\mu^{\mathrm{alg}}(A_{f,j})=0$ and
\begin{equation}\label{eq:odd-alg-lambda}
\lambda^{\mathrm{alg}}(A_{f,j})
=
\lambda\bigl(G_{\chi^{-1}\omega_p^j}\bigr)
+
\lambda\bigl(G_{\omega_p^{1-j}}\bigr).
\end{equation}
For $j=1$,
$\mu^{\mathrm{alg}}(A_{f,1})=0$ and
\begin{equation}\label{eq:odd-alg-lambda-exceptional}
\lambda^{\mathrm{alg}}(A_{f,1})
=
\lambda\bigl(G_{\chi^{-1}\omega_p}\bigr)-1.
\end{equation}
\end{cor}

\begin{proof}
If $j\neq1$, this follows from
Proposition~\ref{prop:odd-selmer-f-corrected},
\eqref{eq:odd-char-p-ramified}, \eqref{eq:odd-char-unramified}, and the theorem of
Ferrero--Washington.

For $j=1$, by \eqref{eq:odd-selmer-exceptional-iso},
Lemma~\ref{lem:odd-class-field}, and \eqref{eq:odd-char-p-ramified}, we have
\begin{equation}\label{eq:odd-char-ideal-exceptional}
\operatorname{Char}_{\Lambda}
\operatorname{Sel}(\mathbf Q_\infty,A_{f,1})^\vee
=
\left(
\frac{
G_{\chi^{-1}\omega_p}(X)
}{
1+X-u
}
\right).
\end{equation}
The divisibility also follows directly from
\[
G_{\chi^{-1}\omega_p}(u-1)
=
L_p(0,\chi^{-1}\omega_p)
=
0.
\]
Since $1+X-u$ is distinguished of degree one, the desired formulas
follow.
\end{proof}

\begin{thm}[The weight-one main conjecture on the odd branches]
\label{thm:odd-main-f}
For every odd $1\leq j\leq p-2$,
\begin{equation}\label{eq:odd-main-f}
  \cL(f,\omega_p^j)\Lambda
=
\operatorname{Char}_{\Lambda}
\operatorname{Sel}(\mathbf Q_\infty,A_{f,j})^\vee.
\end{equation}
\end{thm}

\begin{proof}
If $j\neq1$, Theorem~\ref{p_adic_L_wt_one_minus} gives
\[
  \cL(f,\omega_p^j)
\sim_{\mathcal O}
G_{\chi^{-1}\omega_p^j}(X)
G_{\omega_p^{1-j}}
\bigl(u(1+X)^{-1}-1\bigr),
\]
whereas Proposition~\ref{prop:odd-selmer-f-corrected} and the abelian
main conjecture give
\[
\operatorname{Char}_{\Lambda}
\operatorname{Sel}(\mathbf Q_\infty,A_{f,j})^\vee
=
\left(
G_{\chi^{-1}\omega_p^j}(X)
G_{\omega_p^{1-j}}
\bigl(u(1+X)^{-1}-1\bigr)
\right).
\]
Thus the two principal ideals agree.

At $j=1$, equation~\eqref{eq:odd-analytic-exceptional} gives
\[
  \cL(f,\omega_p)(X)
\sim_{\mathcal O}
\mathcal V_1(X)\frac{G_{\chi^{-1}\omega_p}(X)}{1+X-u}.
\]
Since $\mathcal V_1(X)$ is an Iwasawa unit, this series generates
the characteristic ideal in \eqref{eq:odd-char-ideal-exceptional}.
\end{proof}

\subsubsection{Higher-weight specializations}

\begin{prop}\label{prop:odd-algebraic-g}
For every odd $1\leq j\leq p-2$, one has
\[
\mu^{\mathrm{alg}}(A_{g,j})=0
\]
and
\begin{equation}\label{eq:odd-alg-g}
\lambda^{\mathrm{alg}}(A_{g,j})
=
\lambda^{\mathrm{alg}}(A_{f,j}).
\end{equation}
\end{prop}

\begin{proof}
Let $S_{\fk,j}$ be the residual Selmer group defined over $\fk$ from
the representation
$\overline H\otimes\omega_p^j$ and the ordinary line
$\fk(ae_-+be_+)$. Put $\F_f=\cO_f/\pi_f$ and
$\F_g=\cO_g/\pi_g$. Proposition~\ref{residual_repr_is_semisimple}
and flat scalar extension give
\begin{equation}\label{eq:odd-residual-selmer-identification}
\begin{aligned}
\operatorname{Sel}(\mathbf Q_\infty,A_{f,j}[\pi_f])
&\simeq S_{\fk,j}\otimes_{\fk}\F_f,\\
\operatorname{Sel}(\mathbf Q_\infty,A_{g,j}[\pi_g])
&\simeq S_{\fk,j}\otimes_{\fk}\F_g.
\end{aligned}
\end{equation}
In particular, if $d_j:=\dim_{\fk}S_{\fk,j}$, then both residual
Selmer groups have dimension $d_j$ over their respective residue
fields.

Since $j\neq0$, Proposition~\ref{SES1} gives
\begin{equation}\label{eq:odd-residual-usual-selmer}
\operatorname{Sel}
(\mathbf Q_\infty,A_{g,j}[\pi_g])
\simeq
\operatorname{Sel}
(\mathbf Q_\infty,A_{g,j})[\pi_g],
\end{equation}
and the analogous statement for $f$.

Write $M_f=\Sel(\Q_\infty,A_{f,j})^\vee$ and
$M_g=\Sel(\Q_\infty,A_{g,j})^\vee$. By
Corollary~\ref{cor:odd-algebraic-f}, $M_f$ is torsion and has $\mu=0$,
so it is finitely generated over $\cO_f$, of rank $\lambda^{\alg}(A_{f,j})$.
In particular, $d_j=\dim_{\F_f}(M_f/\pi_f M_f)$ is finite.
The residual comparison gives $\dim_{\F_g}(M_g/\pi_g M_g)=d_j$;
thus $M_g$ is finitely generated over $\cO_g$ and has $\mu=0$.
Both duals are now known to be torsion over their Iwasawa algebras,
so Proposition~\ref{sel_no_fin_submod} applies at both weights.
Since their $\cO_f$- and $\cO_g$-torsion submodules are finite,
these submodules vanish. Both ranks consequently equal $d_j$, giving
\eqref{eq:odd-alg-g}.
\end{proof}

\begin{thm}[The higher-weight main conjecture on the odd branches]
\label{thm:odd-main-g}
Suppose that $g$ is the ordinary $p$-stabilization of a newform of
level $N$. For every odd $1\leq j\leq p-2$,
\begin{equation}\label{eq:odd-main-g}
  \cL(g,\omega_p^j)\Lambda_{\mathcal O_g}[1/p]
=
\operatorname{Char}_{\Lambda_{\mathcal O_g}[1/p]}
\operatorname{Sel}(\mathbf Q_\infty,A_{g,j})^\vee.
\end{equation}
\end{thm}

\begin{proof}
Theorem~\ref{thm:odd-main-f}, Proposition~\ref{prop:odd-algebraic-g},
and Corollary~\ref{ana_inv_(g,j)} give
\begin{equation}\label{eq:odd-invariant-equality}
 \lambda^{\alg}(A_{g,j})=\lambda^{\alg}(A_{f,j})
 =\lambda^{\an}(f,j)=\lambda^{\an}(g,j).
\end{equation}
By \eqref{eq:divisible_kato}, the algebraic characteristic series
divides the analytic series after inverting $p$. The quotient has
$\lambda$-invariant zero by \eqref{eq:odd-invariant-equality}, so
it is a unit in $\Lambda_{\cO_g}[1/p]$.
\end{proof}
\appendix
\section{Arithmetic criteria for cuspidal Gorensteinness}
\label{app:maksoud-gorenstein}

The Gorenstein criteria and complete-intersection results recalled in
this appendix are entirely due to Maksoud \cite{maksoud-exceptional}.
We restate his quadratic criterion in the notation of this paper.
Put $R=\gh_{\theta,\gm}$. The correspondence
between ordinary and dual Hecke operators, together with the change of
weight coordinate, identifies $R$ with Maksoud's local cuspidal Hecke
algebra \cite[proof of Theorem~3.5.2]{maksoud-exceptional}.
Thus his structural results apply to the ring used throughout this paper.

Let $K_\chi$ be the field cut out by $\chi$, and define
\[
 A'_{K_\chi}:=
 \left(\operatorname{Cl}(K_\chi)/
 \langle[\mathfrak q]\mid\mathfrak q\mid p\rangle\right)
 \otimes_{\Z}\Z_p.
\]
\begin{samepage}
\begin{thm}[Maksoud, quadratic case]
\label{thm:maksoud-quadratic}
Suppose $\chi^2=1$. Then
\[
 R\text{ is Gorenstein}
 \quad\Longleftrightarrow\quad A'_{K_\chi}=0.
\]
When these conditions hold, $R$ is a local complete intersection.
\end{thm}
\begin{proof}
Apply \cite[Theorems~3.3.2 and~3.4.2]{maksoud-exceptional} through the
Hecke-algebra identification above.
\end{proof}
\end{samepage}

In particular, suppose $\chi$ is quadratic and $p$ does not divide the
class number of $K_\chi$. The $p$-primary part of its class group is then
zero, so $A'_{K_\chi}=0$ and Theorem~\ref{thm:maksoud-quadratic} applies.
The same conclusion holds whenever the classes of primes above $p$
generate the $p$-primary part of the class group. This gives a direct
class-group test for the hypothesis of our analytic theorems.

To explain the expected prevalence of this condition, fix $p\geq5$ and
let $K$ vary over imaginary quadratic fields, ordered by absolute
discriminant. Writing $h_K=\#\operatorname{Cl}(K)$, the Cohen--Lenstra
heuristics \cite[Section~8]{cohen-lenstra1984} predict that the proportion
of fields with $p\nmid h_K$ is
\[
 c_p=\prod_{i=1}^{\infty}(1-p^{-i}).
\]
In particular, $c_5=0.7603327958\ldots$, and $c_p\geq c_5>0.76$ for
every $p\geq5$. Thus the sufficient condition on the class number is predicted
to hold for a majority of imaginary quadratic fields for every fixed
$p$ in our range.

The local-condition refinement predicts the same class-group
distribution after imposing the condition that $p$ split in $K$
\cite[Conjecture~1.4]{wood-local-conditions}. Since $p\nmid h_K$ implies
$A'_K=0$, this predicts that a majority of split quadratic fields
satisfy Maksoud's criterion. Applying the same class-group model
within the conductor-restricted subfamily $p\nmid N\phi(N)$ therefore
suggests that the cuspidal Gorenstein hypothesis holds in most quadratic
cases considered in this paper. This assertion about prevalence is
heuristic.

\begin{ex}\label{ex:maksoud-gorenstein}
Take $p=5$ and let $\chi$ be the quadratic character of
$K_\chi=\Q(\sqrt{-79})$. Then $N=79$, $5\nmid N\phi(N)$, and
$\chi(5)=1$. The LMFDB records that its class group is cyclic of order
$5$ \cite[Number field~2.0.79.1]{lmfdb-minus79}. We verify Maksoud's
criterion directly. A prime above $5$ is nonprincipal,
since its principality would give an element of norm $5$, whereas
\[
 \operatorname{Norm}_{K_\chi/\Q}
 \left(a+b\frac{1+\sqrt{-79}}{2}\right)
 =a^2+ab+20b^2
 =\left(a+\frac b2\right)^2+\frac{79b^2}{4}
\]
cannot equal $5$ for integers $a,b$. Thus its ideal class generates the
class group, $A'_{K_\chi}=0$, and $R$ is Gorenstein. This example satisfies
all the standing conditions on $p,N,\chi$ and has a nontrivial
$p$-class group.
\end{ex}

\bibliography{biblio}
\bibliographystyle{alpha}
\end{document}